\documentclass[11pt, draft]{amsart}
\usepackage{amssymb, amstext, amscd, amsmath, amssymb}
\usepackage{mathtools, xypic, paralist, color, dsfont, rotating}
\usepackage{verbatim}
\usepackage{enumerate,enumitem}
\usepackage{float}
\usepackage{setspace}
\usepackage{pifont}

\usepackage{tikz}
\usetikzlibrary{fit,positioning,arrows,automata,calc}
\tikzset{
  main/.style={circle, minimum size = 30pt, thick, draw = black!80, node distance = 10mm},
  connect/.style={-latex, thick},
  box/.style={rectangle, draw = white!100}
}
\usepackage{tikz-cd}

\usepackage{tikz}

\floatstyle{boxed} 
\restylefloat{figure}

\numberwithin{equation}{section}
\usepackage[a4paper]{geometry}
\usepackage{quoting}
\quotingsetup{vskip=.1in}
\quotingsetup{leftmargin=.6in} 
\quotingsetup{rightmargin=.6in}

\let\OLDthebibliography\thebibliography
\renewcommand\thebibliography[1]{
  \OLDthebibliography{#1}
  \setlength{\parskip}{0pt}
  \setlength{\itemsep}{2pt plus 0.5ex}
}

\makeatletter
\def\@cite#1#2{{\m@th\upshape\bfseries%
[{#1\if@tempswa{\m@th\upshape\mdseries, #2}\fi}]}}
\makeatother
\theoremstyle{plain}
\newtheorem{theorem}{Theorem}[section]
\newtheorem{corollary}[theorem]{Corollary}
\newtheorem{proposition}[theorem]{Proposition}

\theoremstyle{definition}

\newtheorem{remark}[theorem]{Remark}

\newtheorem*{acknow}{Acknowledgements}
\newtheorem*{open}{Open Access Statement}

\theoremstyle{remark}

\mathtoolsset{centercolon}
  \newcommand{\B}{{\mathcal{B}}}
  \newcommand{\C}{{\mathcal{C}}}

  \newcommand{\F}{{\mathcal{F}}}

  \newcommand{\I}{{\mathcal{I}}}
  \newcommand{\J}{{\mathcal{J}}}
  \newcommand{\K}{{\mathcal{K}}}
\renewcommand{\L}{{\mathcal{L}}}
  
  \newcommand{\N}{{\mathcal{N}}}
\renewcommand{\O}{{\mathcal{O}}}
\renewcommand{\P}{{\mathcal{P}}}

\renewcommand{\S}{{\mathcal{S}}}
  \newcommand{\T}{{\mathcal{T}}}

\newcommand{\eps}{\varepsilon}
\def\al{\alpha}
\def\be{\beta}

\def\De{\Delta}
\def\de{\delta}

\def\io{\iota}

\def\la{\lambda}

\def\si{\sigma}

\newcommand\vphi{\varphi}

\newcommand{\bC}{\mathbb{C}}

\newcommand{\bN}{\mathbb{N}}

\newcommand{\bZ}{\mathbb{Z}}
\newcommand{\bR}{\mathbb{R}}
\newcommand{\fA}{{\mathfrak{A}}}

\newcommand{\fH}{{\mathfrak{H}}}
\newcommand{\fh}{{\mathfrak{h}}}

\newcommand{\Bx}{{\mathbf{x}}}
\newcommand{\By}{{\mathbf{y}}}

\newcommand{\foral}{\text{ for all }}
\newcommand{\qand}{\quad\text{and}\quad}

\newcommand{\qiff}{\quad\text{if and only if}\quad}
\newcommand{\qfor}{\quad\text{for}\quad}

\newcommand{\ca}{\mathrm{C}^*}

\newcommand{\cenv}{\mathrm{C}^*_{\textup{env}}}

\newcommand{\ol}{\overline}

\newcommand{\ad}{\operatorname{ad}}

\newcommand{\alg}{\operatorname{alg}}

\newcommand{\cov}{\operatorname{c}}

\newcommand{\fock}{\operatorname{F}}

\newcommand{\id}{{\operatorname{id}}}

\newcommand{\mt}{\emptyset}

\newcommand{\scv}{\operatorname{sc}}
\newcommand{\spn}{\operatorname{span}}

\newcommand{\sumoplus}{\operatornamewithlimits{\sum \strut^\oplus}}

\newcommand{\sca}[1]{\left\langle#1\right\rangle} 
\newcommand{\nor}[1]{\left\Vert #1\right\Vert} 
\newcommand{\bo}[1]{\mathbf{#1}} 

\newcommand{\tes}[7]{
	\xymatrix@C=2cm@R=1.5cm{
		K_0\left(#1\right) \ar[r]^{#2} & K_0\left(#3\right) \ar[r]^{#4} & K_0\left(#5\right) \ar[d] \\
		K_1\left(#5\right) \ar[u] & K_1\left(#3\right) \ar[l]^{#7} & K_1\left(#1\right) \ar[l]^{#6}
	}
}

\addtocontents{toc}{\protect\setcounter{tocdepth}{1}}

\begin{document}

\title[Fock covariance and the reduced Hao--Ng isomorphism problem]{On Fock covariance for product systems and the reduced Hao--Ng isomorphism problem by discrete actions}

\author[E.T.A. Kakariadis]{Evgenios T.A. Kakariadis}
\address{School of Mathematics, Statistics and Physics\\ Newcastle University\\ Newcastle upon Tyne\\ NE1 7RU\\ UK}
\email{evgenios.kakariadis@newcastle.ac.uk}

\author[I.A. Paraskevas]{Ioannis Apollon Paraskevas}
\address{Department of Mathematics\\ National and Kapodistrian University of Athens\\ Athens\\ 1578 84\\ Greece}
\email{ioparask@math.uoa.gr}



\begin{abstract}
We provide a characterisation of equivariant Fock covariant injective representations for product systems.
We show that this characterisation coincides with Nica covariance for compactly aligned product systems over right LCM semigroups of Kwa\'{s}niewski and Larsen, and with the Toeplitz representations of a discrete monoid of Laca and Sehnem.
By combining with the framework established by Katsoulis and Ramsey, we resolve the reduced Hao--Ng isomorphism problem for generalised gauge actions by discrete groups.
\end{abstract}

\maketitle

\section{Introduction}

\subsection{Fock covariance}

Hilbertian representations play an important role in the study of algebraic structures and symmetries.
One of the cornerstone examples is the Gel'fand--Raikov Theorem that ``identifies'' a locally compact group with its unitary representations, while further examples arise naturally in the context of invertible transformations on a (possibly noncommutative) state space. 
In this endeavour, algebraic structures associated with discrete groups have been successfully unified under the theory of Fell bundles.
Product systems over a unital subsemigroup $P$ of a discrete group $G$ provide a common context to describe irreversible transformations as a semigroup analogue of Fell bundles.
Examples include subshift spaces, rank one or higher rank graphs, topological graphs, and semigroup actions on C*-algebras (in particular in relation to transfer operators), to mention but a few beyond the Fell bundles setup.

Product systems were introduced by Arveson \cite{Arv89} disguised under duality, and later put in context by Dinh \cite{Din91} for discrete subsemigroups of $\bR_+$.
Nica \cite{Nic92} considered operator algebras related to semigroups that admit a type of order structure, i.e., quasi-lattice semigroups, for which a representation by left translations was well-defined, thus making the resulting C*-algebras amenable to study.
It had been known that, unlike to the group case, taking plain isometric representations leads to intractable objects, e.g., the universal isometric C*-algebra of $\bN \times \bN$ is not even nuclear \cite{Mur87}.
A richer structure is at hand when one considers the relations in the Fock representation of a quasi-lattice order, now commonly known as Nica covariance.
Nica's setup is now covered by the theory as a product system with \lq\lq trivial'' fibers.
On the other hand, Pimsner \cite{Pim95} introduced C*-algebras of product systems over $\bZ_+$, i.e., C*-correspondences, that encodes successfully graph C*-algebras, crossed products over $\bZ$, or partial transformations, see for example \cite{Kat03}.
Motivated by \cite{Nic92, Pim95} as well as his work with Raeburn \cite{FR98}, Fowler \cite{Fow02} proceeded to an in-depth study of product systems over quasi-lattices, that inspired a great number of subsequent works.
In order to make full use of this additional feature, Fowler imposed several axioms on the product system resulting to a by-default  Wick ordering.
Moving beyond the quasi-lattice order structure, Brownlowe, Larsen and Stammeier \cite{BLS18} considered product systems related to dynamical systems over right LCM semigroups. 
Kwa\'{s}niewski and Larsen \cite{KL19a, KL19b} extended this setup to general product systems over right LCM semigroups, and essentially this is the furthest one may go and still have a Wick ordering.
Cuntz, Deninger and Laca \cite{CDL13} used the Fock model beyond Nica's work \cite{Nic92} for $P = R \rtimes R^\times$, where Nica covariance may fail, but one can still use the Fock representation as the prototype for covariant relations.

The coaction of the ambient group $G$ has been pivotal in the theory of product systems, in particular in relation to boundary quotients by Carlsen, Larsen, Sims and Vittadello \cite{CLSV11} and Sehnem \cite{Seh18, Seh21}, as well as in semigroup algebras by Laca and Sehnem \cite{LS21} and the first author, Katsoulis, Laca and Li \cite{KKLL21a}.
Dor-On, the first author, Katsoulis, Laca and Li \cite{DKKLL20} noted that there is a canonical Fell bundle arising from the normal coaction on $\T_\la(X)$, denoted here by $\F \C X$.
It appears that, while $\T_\la(X)$ is the reduced C*-algebra of $\F\C X$, it is not always the reduced C*-algebra for the universal C*-algebra $\T(X)$ of the representations of $X$ (as it happens for $P=\bZ_+$).
To make a distinction, we will write $\T_{\cov}^{\fock}(X)$ for the universal C*-algebra of $\F\C X$; a representation of $X$ that promotes to a $*$-representation of $\T_{\cov}^{\fock}(X)$ will be called \emph{Fock covariant}.
It was shown in \cite{DKKLL20} that Fock covariant representations coincide with Nica covariant representations for a compactly aligned product system $X$ over a right LCM semigroup $P$, and thus $\T_{\cov}^{\fock}(X)$ coincides with the universal Nica covariant C*-algebra $\N\T(X)$ in this case.

The difference between $\T(X)$ and $\T_{\cov}^{\fock}(X)$ is apparent even when every $X$ arises from $P$ by one-dimensional spaces, where the unital representations of $X$ are in bijection with the unital semigroup representations of $P$.
In his ground-breaking work, Li \cite{Li12} took motivation from \cite{CDL13} and revived the interest on semigroup representations with the twist that, apart from the semigroup structure, they remember the principal ideals and their intersections, i.e., the \emph{constructible} ideals.
Li's models are variants of the relations in the Fock representation and were further examined by the first author, Katsoulis, Laca and Li \cite{KKLL21a}, also in connection with inverse semigroup realizations by Norling \cite{Nor14}.
At the same time Laca and Sehnem \cite{LS21} identified completely the relations that promote a general semigroup representation to a Fock covariant representation of its related product system.
It is thus natural to ask for the characterisation of Fock covariance for general product systems.

Motivated by the Hao--Ng isomorphism problem, our first aim in this work is to provide a description of the equivariant Fock covariant injective representations of $X$.
For our purposes, we will view a product system $X$ as a family of C*-correspondences $\{X_p\}_{p \in P}$ in a common $\B(H)$ that satisfies some natural semigroup relations, i.e., that $A := X_{e}$ is a C*-algebra, and
\begin{equation} \label{eq:defn ps}
X_p \cdot X_q \subseteq X_{pq}
\text{ and }
X_p^* \cdot X_{pq} \subseteq X_q
\foral 
p, q \in P.
\end{equation}
This setup enables to still define left creation operators on the full Fock space $\F X$ giving rise to the Fock C*-algebra $\T_\la(X)$, while it recovers Fowler's product systems in \cite{Fow02} which arise under the stronger saturation condition $[X_p\cdot X_q]=X_{pq}$ for all $p,q \in P$.
We note that most of the works for product systems assume non-degeneracy of the left action on $X$.
Nevertheless, non-degeneracy is rarely used, although imposed by the saturation condition of Fowler \cite{Fow02}.
Our setting from (\ref{eq:defn ps}) contains all non-degenerate product systems, while accommodating possible degenerate ones.
Such cases arise naturally for example in C*-dynamical systems as the ones considered by  Davidson, Fuller and the first named author in \cite{DFK17}.

The characterisation of Fock covariance that we provide is given in terms of the $\bo{K}_{\Bx}$-cores for $\Bx \in \J$.
We then use the natural description of the fixed point algebra by an inductive limit of C*-algebras on the $\cap$-closed finite subsets of the constructible ideals of $P$ in $\J$, which are spanned by the $\bo{K}_{\bullet}$-cores.

\medskip

\noindent
{\bf Theorem A.} \emph{(Theorem \ref{T:Fock cov})
Let $P$ be a unital subsemigroup of a discrete group $G$ and let $X$ be a product system over $P$.
Let $\hat{t}$ be a representation of $X$ such that $\T(X) = \ca(\hat{t})$.
An equivariant injective representation $t$ of $X$ is Fock covariant if and only if $t$ satisfies the following conditions:
\begin{enumerate}
\item $\bo{K}_{\mt, t_\ast} = (0)$.
\item For any $\cap$-closed $\F=\{\bo{x}_1,\dots,\bo{x}_n\} \subseteq \J$ such that $\bigcup_{i=1}^n \bo{x}_i \neq \mt$, and any $b_{\bo{x}_i}\in \bo{K}_{\bo{x}_i,\hat{t}_\ast}$, with $i=1,\dots,n$, the following property holds:
\[
\text{if $\sum\limits_{i: r\in \bo{x}_i}t_\ast(b_{\bo{x}_i})t_r(X_r)=(0)$ for all $r\in \bigcup_{i=1}^n \bo{x}_i$, then $\sum\limits_{i=1} ^n t_\ast(b_{\bo{x}_i})=0$.}
\]
\end{enumerate}}

We square this characterisation with established results from the literature.
First, we provide a direct connection with the equivariant Nica covariant injective representations when $P$ is a right LCM semigroup which recovers \cite[Proposition 4.3]{DKKLL20}.
Secondly, when each $X_p = \bC$ then we are at the setting of the semigroup representations and our description gives back the relations identified by Laca and Sehnem in \cite[Section 3]{LS21}; see also \cite[Remark 3.15]{KKLL21b} (note that injectivity is automatic in the latter case).
The connection with \cite{DKKLL20} and \cite{LS21} can be seen categorically, however we give independent proofs from first principles that validate the compatibility with our results herein.

As we will see, a description of Fock covariance just for the case of equivariant representations that are injective on $X$ is enough to tackle key questions in the theory.
One application of particular interest, and our main motivation for this work, concerns the reduced Hao--Ng isomorphism problem \cite{HN08} which we describe below.

\subsection{The Hao--Ng isomorphism problem}

Along with $\T(X)$ and $\T_{\cov}^{\fock}(X)$ there is a significant boundary quotient.
For the case $P=\bZ_+$ this is the Cuntz--Pimsner algebra $\O_X$, which generalises the graph C*-algebra and the Cuntz algebra, and was provided in full generality by Katsura \cite{Kat04} following the work of many hands.
Katsura \cite{Kat07} has also proven that $\O_X$ is the terminal object for the equivariant injective representations of $X$ (Fock covariance is automatically satisfied in this case), while Katsoulis and Kribs \cite{KK06} have shown that $\O_X$ is the C*-envelope of the tensor algebra of $X$.
Carlsen, Larsen, Sims and Vittadello \cite{CLSV11} have provided the existence of the terminal object in the case of a quasi-lattice, while Dor-On, the first author, Katsoulis, Laca and Li \cite{DKKLL20} have tackled the case for right LCM semigroups in terms of a terminal object for the coaction on the tensor algebra.
A pivotal step in the general direction was established by Sehnem \cite{Seh18} where the appropriate universal quotient $A \times_X P$ was coined.
The key property of the equivariant representations of $A \times_X P$ is that injectivity on $X$ implies injectivity on the fixed point algebra.
Equivariance allows to view $A \times_X P$ as the universal C*-algebra of the \emph{strong covariant bundle} denoted by $\S\C X$ in \cite{DKKLL20}.
The terminal object in \cite{DKKLL20} is then the reduced C*-algebra of $\S\C X$ denoted by $A \times_{X, \la} P$.
Sehnem \cite{Seh21} has shown that completely isometric representations of the tensor algebra $\T_\la(X)^+$ admit automatically a conditional expectation.
With this remarkable result at hand, Sehnem \cite{Seh21} resolved the old standing problem of identifying $A \times_{X, \la} P$ with the C*-envelope of $\T_\la(X)^+$, encompassing the work of many authors.

The reduced Hao--Ng isomorphism problem refers to a group action $\al$ of $\fH$ on $\T_\la(X)$ that leaves invariant every $\la_p(X_p)$; such an action is called a \emph{generalised gauge action}.
In this case there is an induced product system denoted by $X \rtimes_{\al, \la} \fH$ and an induced action $\dot{\al}$ on $A \times_{X,\la} P$.
The main question is whether there is a canonical $*$-isomorphism
\begin{equation} \label{eq:sc cp}
(A \rtimes_{\al, \la} \fH) \times_{X \rtimes_{\al, \la} \fH, \la} P
\stackrel{?}{\simeq} 
(A \times_{X, \la} P) \rtimes_{\dot{\al}, \la} \fH,
\end{equation}
i.e., if the reduced strong covariant functor commutes with the reduced crossed product functor.
The problem was introduced in \cite{HN08}, where it was solved for $P= \bZ_+$ and $\fH$ being an amenable locally compact group.
As applications of their results, Hao and Ng recover previous results on Hilbert bimodules by Abadie \cite{Aba07}, and on graph C*-algebras by Kumjian and Pask \cite{KP99} without using groupoid C*-algebras.
The importance of the Hao--Ng isomorphism problem has been further emphasised by Katsoulis \cite{Kat17} in connection to the work of Echterhoff, Kaliszewski, Quigg and Raeburn \cite{EKQR06} on imprimitivity theorems for C*-dynamical systems.

In a series of works, Katsoulis \cite{Kat17, Kat20}, and Katsoulis and Ramsey \cite{KR21}, have rephrased the reduced Hao--Ng isomorphism problem in terms of C*-envelopes and crossed products \cite{KR16}, i.e., whether there exists a canonical $*$-isomorphism
\begin{equation} \label{eq:cenv cp}
\cenv(\T_\la(X \rtimes_{\al, \la} \fH)^+) \stackrel{?}{\simeq} \cenv(\T_\la(X)^+) \rtimes_{\dot{\al}, \la} \fH.    
\end{equation}
The significance of this approach is manifested in \cite{Kat17} where the problem was resolved for $P = \bZ_+$ and $\fH$ discrete.
The reduced Hao--Ng isomorphism problem was later answered when $P$ defines an abelian lattice order in $G$ by Dor-On and Katsoulis \cite{DK20} when $\fH$ is discrete, and by Katsoulis \cite{Kat20} when $\fH$ is locally compact abelian.
Furthermore it has been answered when $P$ is a right LCM semigroup and $\fH$ is discrete by Dor-On, the first author, Katsoulis, Laca and Li \cite{DKKLL20}.
The question remains open for a generalised gauge action by a general group $\fH$.

The main approach in \cite{DKKLL20, DK20, Kat17, Kat20, KR21} has been to use the independence condition for right LCM semigroups and Nica covariance of the identity representation 
\begin{equation}
X \rtimes_{\al, \la} \fH \hookrightarrow \T_\la(X) \rtimes_{\al, \la} \fH
\end{equation}
in order to obtain a canonical $*$-isomorphism
\begin{equation}\label{eq:fc cp}
\T_\la(X \rtimes_{\al, \la} \fH) \simeq \T_\la(X) \rtimes_{\al, \la} \fH.
\end{equation}
From there it follows that
\begin{equation}\label{eq:ten cp}
\T_\la(X \rtimes_{\al, \la} \fH)^+ \simeq \T_\la(X)^+ \rtimes_{\al, \la} \fH,
\end{equation}
and then the C*-envelope theory of Katsoulis and Ramsey \cite{KR16} can be implemented.
In the absence of independence, it is unclear whether the $*$-isomorphism (\ref{eq:fc cp}) still holds, but it does not exclude the possibility that the completely isometric isomorphism (\ref{eq:ten cp}) is valid.
A careful investigation of the arguments of \cite{DKKLL20, DK20, Kat17, Kat20, KR21} reveals that Fock covariance of the identity representation $X \rtimes_{\al,\la} \fH \hookrightarrow \T_\la(X) \rtimes_{\al,\la} \fH$ enables the completely isometric isomorphism (\ref{eq:ten cp}).
Our characterisation applies towards this resolution when $\fH$ is discrete without restrictions on the product system $X$ and the semigroup $P$.
This recovers the aforementioned results appearing in \cite[Theorem 6.3]{DKKLL20}, \cite[Theorem 6.6]{DK20}, and \cite[Theorem 3.2]{Kat17}.

\medskip

\noindent
{\bf Theorem B.} \emph{(Theorem \ref{T:hao-ng})
Let $P$ be a unital subsemigroup of a discrete group $G$ and let $X$ be a product system over $P$.
Let $\al$ be a generalised gauge action of a discrete group $\fH$ on $\T_\la(X)$.
Then the identity representation $X\rtimes_{\al,\la}\fH \hookrightarrow \T_{\la}(X) \rtimes_{\al,\la} \fH$ lifts to a completely isometric isomorphism
\[
\T_\la(X \rtimes_{\al, \la} \fH)^+ \simeq \T_\la(X)^+ \rtimes_{\al, \la} \fH.
\]
Consequently, the reduced Hao--Ng isomorphism problem has an affirmative answer, i.e.,
\[
(A \rtimes_{\al, \la} \fH) \times_{X \rtimes_{\al, \la} \fH,\la} P
\simeq
(A \times_{X, \la} P) \rtimes_{\dot \al, \la} \fH,
\]
by a canonical $*$-isomorphism, where $\dot{\al}$ is the induced action of $\fH$ on $A \times_{X,\la} P$.}

\medskip

The structure of the manuscript is as follows.
In Section \ref{S:pre} we fix notation and provide some relevant results we will be using.
In Section \ref{S:fock covariant} we give the description of Fock covariance, and applications to compactly aligned product systems over right LCM semigroups in connection to \cite{DKKLL20, KL19a, KL19b}, as well as to semigroup representations in connection to \cite{LS21}.
In Section \ref{S:Hao-Ng} we provide the context for the reduced Hao--Ng isomorphism problem, and then proceed to its resolution when $\fH$ is discrete.
Throughout, we comment when non-degeneracy is not required, and we provide details when a different proof is needed, when using results from the literature.

\begin{acknow}
The authors acknowledge that this research work was supported within the framework of the National Recovery and Resilience Plan Greece 2.0, funded by the European Union - NextGenerationEU (Implementation Body: HFRI. Project name: Noncommutative Analysis: Operator Systems and Nonlocality. HFRI Project Number: 015825).
The second author acknowledges that this research work was supported by the Hellenic Foundation for Research and Innovation (HFRI) under the 5th Call for HFRI PhD Fellowships (Fellowship Number: 19145).
The second author acknowledges that this publication is based upon work from COST Action CaLISTA CA21109 supported by COST (European Cooperation in Science and Technology), www.cost.eu.
This manuscript will form part of the second author's PhD thesis.

The authors would like to thank Elias Katsoulis for his helpful comments and remarks on the draft of the manuscript.
The authors would like to thank the referee for the comments and remarks that helped improve the presentation.
\end{acknow}

\begin{open}
For the purpose of open access, the authors have applied a Creative Commons Attribution (CC BY) license to any Author Accepted Manuscript (AAM) version arising.
\end{open}

\section{Preliminaries} \label{S:pre}

We begin with some preliminaries on coactions on C*-algebras, semigroup C*-algebras and product systems to fix notation.
All groups and semigroups we consider in this section are discrete.
We will write $\otimes$ for the minimal tensor product between C*-algebras.
A map between algebras will be called canonical if it preserves generators of the same index.
For notational convenience, we will write $x^0 = 1$ for an element $x$ in a unital algebra with unit $1$.
If $E$ is a subset of a normed linear space $F$ then we will write $[E]$ for the closed linear span generated by $E$ in $F$.

\subsection{Coactions on C*-algebras}

For a discrete group $G$ we write $u_g$ for the generators of the universal group C*-algebra $\ca_{\max}(G)$, and $\la_g := \la(u_g)$ for the left regular representation $\la \colon \ca_{\max}(G) \to \ca_\la(G)$.
We write $\chi$ for the character on $\ca_{\max}(G)$.
By the universal property of $\ca_{\max}(G)$ there exists a faithful $*$-homomorphism
\[
\De \colon \ca_{\max}(G) \to \ca_{\max}(G) \otimes \ca_{\max}(G); u_g \mapsto u_g \otimes u_g.
\]
By Fell's absorption principle there exists a faithful $*$-homomorphism
\[
\De_\la \colon \ca_\la(G) \to \ca_\la(G) \otimes \ca_\la(G); \la_g \mapsto \la_g \otimes \la_g,
\]
with the additional property that $\De_\la \circ \la = (\la \otimes \la) \circ \De$.

There is a direct connection between coactions on C*-algebras, gradings and Fell bundles.
The reader is addressed to \cite{Exe97, Exe17, Qui96} where this theory is laid in full detail.
Since the group $G$ is discrete, the coactions we consider are automatically non-degenerate in the sense of \cite{Qui96}, see for example \cite[Remark 3.2]{DKKLL20}.

We will say that a C*-algebra $\C$ \emph{admits a coaction $\de$ by $G$} if there is a faithful $*$-homomorphism $\de \colon \C \to \C \otimes \ca_{\max}(G)$ such that \emph{the coaction identity}
\[
(\de \otimes \id) \circ \de = (\id \otimes \De) \circ \de
\]
is satisfied.
Since $G$ is discrete, the coaction identity is equivalent to the induced \emph{spectral spaces}
\[
\C_g := \{c \in \C \mid \de(c) = c \otimes u_g\} \foral g \in G,
\]
together spanning a dense subset of $\C$, see the proof of \cite[Proposition 2.6]{{Ng96}}.
If, in addition, the map $(\id \otimes \la) \circ \de$ is faithful, then $\de$ will be called \emph{normal}.
It follows that $\de$ is normal if and only if $\C$ admits a \emph{reduced} coaction by $G$, i.e., 
there is a faithful $*$-homomorphism $\de_\la \colon \C \to \C \otimes \ca_\la(G)$ such that
\[
(\de_\la \otimes \id) \circ \de_\la = (\id \otimes \De_\la) \circ \de_\la.
\]
We note that if $\de \colon \C \to \C \otimes \ca_{\max}(G)$ is a coaction, then 
\[
E:=(\id\otimes E_\la) \circ (\id\otimes \la) \circ \de
\]
defines a conditional expectation on $\C_e$, where $E_\la$ is the faithful conditional expectation of $\ca_\la(G)$. 
It follows that $\de$ is normal if and only if $E$ is faithful.

More generally, a collection $\{\C_g\}_{g \in G}$ of closed subspaces of a C*-algebra $\C$ is called a \emph{C*-grading of $\C$} if:
\begin{enumerate}
\item $\sum_{g \in G} \C_g$ is dense in $\C$;
\item $\C_g \C_h \subseteq \C_{gh}$; and
\item $\C_g^* \subseteq \C_{g^{-1}}$.
\end{enumerate}
If there exists a conditional expectation $E \colon \C \to \C_e$, then the subspaces are linearly independent, see \cite[Theorem 3.3]{Exe97}.
In this case the C*-grading is called \emph{topological}.
By definition, a coaction on a C*-algebra induces a topological C*-grading.

Gradings form the prototypical example of Fell bundles.
A \emph{Fell bundle} over a group $G$ is a collection $\B = \{\B_g\}_{g \in G}$ of Banach spaces, each of which is called a \emph{fiber}, such that:
\begin{enumerate}
\item there are bilinear and associative \emph{multiplication maps} from $\B_g \times \B_{g'}$ to $\B_{gg'}$ such that $\|b_g b_{g'} \| \leq \|b_g\| \, \|b_{g'}\|$;
\item there are conjugate linear \emph{involution maps} from $\B_g$ to $\B_{g^{-1}}$ such that $(b_g^*)^* = b_g$ and $\|b_g^*\| = \|b_g\|$;
\item $(b_g b_{g'})^* = (b_{g'})^* b_g^*$;
\item $\|b_g^* b_g\| = \| b_g\|^2$;
\item $b_g^* b_g \geq 0$ in $\B_e$.
\end{enumerate}
Note that conditions (i)--(iv) imply that $\B_e$ is in fact a C*-algebra and thus condition (v) makes sense.
A \emph{representation} $\Psi$ of a Fell bundle $\B$ over $G$ is a family $\{\Psi_g\}_{g \in G}$ of linear maps each one defined on $\B_g$ such that:
\begin{enumerate}
\item $\Psi_g(b_g) \Psi_h(b_h) = \Psi_{gh}(b_g b_h)$ for all $g, h \in G$; and
\item $\Psi_g(b_g)^* = \Psi_{g^{-1}}(b_g^*)$ for all $g \in G$.
\end{enumerate}
It follows that $\Psi_e$ is a $*$-homomorphism and thus contractive.
A standard C*-trick shows that every $\Psi_g$ is contractive.
We say that a representation $\Psi$ is \emph{injective} if $\Psi_e$ is injective; in this case every $\Psi_g$ is isometric.
A representation $\Psi$ is called \emph{equivariant} if there exists a $*$-homomorphism 
\[
\de \colon \ca(\Psi) \to \ca(\Psi) \otimes \ca(G); \Psi_g(b_g) \mapsto \Psi_g(b_g) \otimes u_g.
\]
It follows that $\de$ is faithful, with a left inverse given by the map $\id \otimes \chi$, and that it satisfies the coaction identity.

We write $\ca_{\max}(\B)$ for the universal C*-algebra with respect to the representations of $\B$ and we write 
\[
\hat{j} \colon \B \to \ca_{\max}(\B)
\]
for the canonical embedding.
We use the same symbol $\Psi$ for the $*$-homomorphism of $\ca_{\max}(\B)$ induced by a representation $\{\Psi_g\}_{g \in G}$ of $\B$ (as $\Psi \circ \hat{j}_g = \Psi_g$).
By universality, we have that $\{\hat{j}_g\}_{g\in G}$ is an equivariant representation of $\B$, and in particular by \cite[Proposition 17.9]{Exe17} we have that the map $\hat{j}_g \colon \B_g \to [\ca_{\max}(\B)]_g$ is an isometric isomorphism.
Hence, any Fell bundle arises as a C*-grading from some C*-algebra.

The left regular representation of  a Fell bundle $\B$ is defined by considering the left creation operators
\[
(\la_g(b_g) \xi)_{g'} = b_g \xi_{g^{-1} g'} \foral b_{g} \in \B,
\]
on $\ell^2(\B) : =\sum_{g \in G}^\oplus \B_g$ seen as the Hilbert module direct sum over $\B_e$.
We write $\ca_\la(\B)$ for the C*-algebra generated by $\la$.

By writing $j \colon \B \to \ell^2(\B)$ for the canonical embedding of each fiber at the corresponding summand of $\ell^2(\B)$, we can define the unitary
\[
W \colon \ell^2(\B) \otimes \ell^2(G) \to \ell^2(\B) \otimes \ell^2(G); j_g(b_g) \otimes \de_{g'} \mapsto j_g(b_g) \otimes \de_{gg'}.
\] 
It follows that $W$ implements a reduced coaction 
\[
\ca_\la(\B) \stackrel{\simeq}{\longrightarrow} \ca(\la_g(b_g) \otimes I) \stackrel{\ad_W}{\longrightarrow} \ca_\la(\B) \otimes \ca_\la(G),
\]
and thus a normal coaction.
By \cite[Theorem 3.3]{Exe97} we have that, if $\Psi$ is an equivariant representation of $\B$ that is injective on $\B_e$, then there are equivariant canonical $*$-epimorphisms
\[
\ca_{\max}(\B) \longrightarrow \ca(\Psi) \longrightarrow \ca_{\la}(\B).
\]
If, in addition, the coaction on $\ca(\Psi)$ is normal, then $\ca(\Psi) \simeq \ca_\la(\B)$.

More generally, let $\Psi$ be an equivariant (possibly non-injective) representation of $\B$, then by the proof of  \cite[Proposition 21.4]{Exe17} (or by using the Fourier transform) we obtain that every $\Psi_g$ has closed range, and hence we have the induced Fell bundle
\[
\Psi(\B) := \{\Psi_g(\B_g)\}_{g \in G}.
\]
Therefore there are equivariant $*$-epimorphisms, making the following diagram
\[
\xymatrix{
\ca_{\max}(\B) \ar[rr] \ar[d] & & \ca_\la(\B) \ar[d] \\
\ca_{\max}(\Psi(\B)) \ar[r] & \ca(\Psi) \ar[r] & \ca_\la(\Psi(\B))
}
\]
commutative, see \cite[Proposition 21.2, Proposition 21.3]{Exe17}.
We will make use of the following folklore fact for Fell bundles.

\begin{proposition} \label{P:eq dgrm}
Let $\B$ be a Fell bundle over a discrete group $G$, and $\Psi_1$ and $\Psi_2$ be equivariant representations of $\B$.
Then
\[
\ker \Psi_1 \cap [\ca_{\max}(\B)]_e \subseteq \ker \Psi_2 \cap [\ca_{\max}(\B)]_e
\]
if and only if 
\[
\ker \Psi_1 \cap [\ca_{\max}(\B)]_g \subseteq \ker \Psi_2 \cap [\ca_{\max}(\B)]_g
\foral
g \in G.
\]
If any of the above holds, and $\Psi_1(\B)$ is the induced Fell bundle in $\ca(\Psi_1)$, then $\Psi_2$ defines a representation of $\Psi_1(\B)$, and thus there is a commutative diagram
\[
\xymatrix{
\ca_{\max}(\B) \ar[rr] \ar[dr] & & \ca(\Psi_2) \\
& \ca_{\max}(\Psi_1(\B)) \ar[ur] &
}
\]
of equivariant $*$-epimorphisms.
\end{proposition}

\begin{proof}
It is obvious that, if the inclusion holds for all $g \in G$, then in particular it holds for $g=e$.
Conversely, suppose that
\[
\ker \Psi_1 \cap [\ca_{\max}(\B)]_e \subseteq \ker \Psi_2 \cap [\ca_{\max}(\B)]_e,
\]
and let $x \in \ker \Psi_1 \cap [\ca_{\max}(\B)]_g$ for $g \in G$.
Then
\[
x^*x \in \ker \Psi_1 \cap [\ca_{\max}(\B)]_e \subseteq \ker \Psi_2 \cap [\ca_{\max}(\B)]_e,
\]
and so $x \in \ker \Psi_2$.
Consequently we derive that $x \in \ker \Psi_2 \cap [\ca_{\max}(\B)]_g$, as required.

For the second part of the proof, the existence of the maps
\[
\ca_{\max}(\B) \to \ca_{\max}(\Psi_1(\B))
\qand
\ca_{\max}(\B) \to \ca(\Psi_2)
\]
follows by the discussion prior to the statement and the universal property of the universal C*-algebras.
Due to the inclusion of the kernels, for every $g \in G$ we have a commutative diagram
\[
\xymatrix{
\B_g \ar[rr]^{\Psi_2|_{\B_g}} \ar[dr]_{\Psi_1|_{\B_g}} & & \Psi_2(\B_g) \\
& \Psi_1(\B_g) \ar[ur]_{\Psi_g} & 
}
\]
such that $\Psi_g(\Psi_1(b_g)) = \Psi_2(b_g)$ for every $b_g \in \B_g$, and $g \in G$.
Since $\Psi_1$ and $\Psi_2$ are representations of $\B$, we get that the collection $\{\Psi_g\}_{g \in G}$ defines a representation $\Psi$ from $\Psi_1(\B)$ to $\ca(\Psi_2)$. 
Hence $\Psi$ promotes to a $*$-representation of $\ca_{\max}(\Psi_1(\B))$ in $\ca(\Psi_2)$.
By definition this map closes the diagram, and the proof is complete.
\end{proof}

We will be interested in graded quotients of C*-algebras of Fell bundles.
If $\de \colon \C \to \C \otimes \ca_{\max}(G)$ is a coaction on a C*-algebra $\C$, then we say that an ideal $\I \lhd \C$ is \emph{induced} if 
\[
\I = \sca{\I \cap [\C]_e}.
\]
In that case, the canonical quotient map $q_\I$ is equivariant, i.e., the coaction $\de$ descends to a coaction on $\C/\I$, see \cite[Proposition 3.11]{Exe97} and \cite[Proposition A.1]{CLSV11} for the full details.
The following proposition is perhaps folklore, and we include a proof for completeness.

\begin{proposition}\label{P:qnt bnd}
Let $\B$ be a Fell bundle over a discrete group $G$ and let $\I\lhd \ca_{\max}(\B)$ be an induced ideal.
Let $q_\I(\B)$ be the induced Fell bundle from the coaction on $\ca_{\max}(\B)/\I$, where $q_\I \colon \ca_{\max}(\B) \to \ca_{\max}(\B)/\I$ is the canonical quotient map.
Then there exists a commutative diagram
\[
\xymatrix{
\ca_{\max}(\B) \ar[rr]^{\Phi} \ar[dr]_{q_{\I}} & & \ca_{\max}(q_\I(\B))  \\
& \ca_{\max}(\B)/\I \ar[ur]_{\Psi} & 
}
\]
of equivariant $*$-epimorphisms such that $\Psi$ is a $*$-isomorphism.
\end{proposition}

\begin{proof}
Since $\I$ is an induced ideal we have that $q_\I$ is equivariant.
Moreover $\ca_{\max}(\B)/\I$ admits a coaction and let $q_\I(\B) := \{[\ca_{\max}(\B)/\I]_g\}_{g\in G}$ be the induced Fell bundle.
In order to make a distinction, we will write 
\[
\hat{j}^{\B} \colon \B \to \ca_{\max}(\B)
\qand
\hat{j}^{q_{\I}(\B)} \colon q_{\I}(\B) \to \ca_{\max}(q_\I(\B))
\]
for the induced embeddings of the corresponding Fell bundles.

First note that the family $\left\{\hat{j}^{q_\I(\B)}_g\circ q_{\I}\circ \hat{j}_g ^{\B}\right\}_{g\in G}$ is a representation of $\B$ and hence there is an induced $*$-epimorphism 
\[
\Phi \colon \ca_{\max}(\B) \to \ca_{\max}(q_\I(\B)).
\]
By the definition of $\Phi$ we have
\[
\I \cap [\ca_{\max}(\B)]_e = \I \cap \hat{j}^{\B}_e(\B_e) \subseteq  \hat{j}^{\B}_e\left(\ker \left( \hat{j}^{q_\I(\B)}_e\circ q_{\I}\circ\hat{j}^{\B}_e \right)\right)\subseteq \ker\Phi,
\]
and thus we have 
\[
\I = \sca{\I \cap [\ca_{\max}(\B)]_e} \subseteq \ker\Phi.
\]
Therefore we obtain a commutative diagram
\[
\xymatrix{
\ca_{\max}(\B) \ar[rr]^{\Phi} \ar[dr]_{q_{\I}} & & \ca_{\max}(q_{\I}(\B))  \\
& \ca_{\max}(\B)/\I \ar[ur]_{\Psi} & 
}
\]
of $*$-epimorphisms.
By checking on the generators we have that the $*$-epimorphisms are equivariant.

On the other hand, since $q_\I(\B)$ is a topological C*-grading of $\ca_{\max}(\B)/\I$, by \cite[Theorem 19.5]{Exe17} there exists a canonical $*$-epimorphism 
\[
\ca_{\max}(q_\I(\B))\to \ca_{\max}(\B)/\I,
\]
which is the inverse of $\Psi$.
Hence $\Psi$ is a $*$-isomorphism, and the proof is complete.
\end{proof}

\subsection{Semigroups}

We will require some elements from the theory of semigroup algebras and right ideals of a semigroup from the work of Li \cite{Li12}, while we also fix notation.
For a unital discrete left-cancellative semigroup $P$, we let the left creation (isometric) operators given by
\[
V_p \colon \ell^2(P) \longrightarrow \ell^2(P) ; \de_s \mapsto \de_{ps},
\]
and we write $\ca_\la(P) := \ca(V_p \mid p \in P)$.
The restriction to the diagonal $\ell^\infty(P)$ induces a faithful conditional expectation on $\ca_\la(P)$.
For a set $Z \subseteq P$ we will write $E_{[Z]}$ for the projection on $[\{\de_p \mid p \in Z\}] \subseteq \ell^2(P)$.
It follows that
\[
E_{[Z_1]} E_{[Z_2]} = E_{[Z_1 \cap Z_2]}
\foral
Z_1, Z_2 \subseteq P.
\]

For a set $Z \subseteq P$ and $p \in P$ we write
\[
pZ:=\{px \mid x \in Z\}
\qand
p^{-1}Z := \{y \in P \mid py \in Z\}.
\]
By definition we have $p^{-1} P = P$.
We write $\J$ for the smallest family of right ideals of $P$ containing $P$ and $\mt$ that is closed under left multiplication, taking pre-images under left multiplication (as in the sense above) and finite intersections, i.e.,
\[
\J := \left\{ \bigcap\limits_{j=1}^N q_{j, n_j}^{-1} p_{j, n_j} \dots q_{j, 1}^{-1} p_{j, 1} P \mid N, n_j \in \bN; p_{j,k}, q_{j,k} \in P \right\} \bigcup \{\mt\}.
\]
The right ideals of $P$ in $\J$ are called \emph{constructible}.
By the proof of \cite[Lemma 3.3]{Li12} we obtain a reduced form for the elements of $\J$, since
\[
q_n^{-1} p_n \dots q_1^{-1} p_1 p_1^{-1} q_1 \dots p_n^{-1} q_n Z = (q_n^{-1} p_n \dots q_1^{-1} p_1 P) \cap Z
\]
for every $p_i, q_i \in P$ and every subset $Z$ of $P$.
Consequently the set of constructible ideals is automatically closed under finite intersections, i.e.,
\[
\J = \left\{ q_n^{-1} p_n \dots q_1^{-1} p_1 P \mid n \in \bN; p_i, q_i \in P \right\} \cup \{\mt\}.
\]
We will write $\Bx, \By$ etc.\ for the elements in $\J$. 

Henceforth we will assume that $P$ is a unital subsemigroup of a discrete group $G$.
In this case we have
\[
q^{-1} p Z = P \cap q^{-1} \cdot p \cdot Z, \text{ for } Z \subseteq P,
\]
and therefore inductively we obtain
\[
q_n^{-1} p_n \dots q_1^{-1} p_1 P
=
P \cap (q_n^{-1} \cdot p_n \cdot P) \cap (q_n^{-1} \cdot p_n \cdot q_{n-1}^{-1} \cdot p_{n-1} \cdot P) \cap \cdots \cap (q_n^{-1} \cdot p_n \cdots q_1^{-1} \cdot p_1 \cdot P).
\]
Note also that for every $\Bx\in\J$ we can pick $p_1,q_1,\dots,p_n,q_n\in P$ such that $p_1^{-1} q_1\cdots p_n^{-1}q_n =e$ and $\Bx=q_n^{-1} p_n \dots q_1^{-1} p_1 P$.
Moreover, $\ca_\la(P)$ admits a normal coaction by $G$ such that
\[
[\ca_\la(P)]_g = \ol{\spn}\{V_{p_1}^* V_{q_1} \cdots V_{p_n}^* V_{q_n} \mid n \in \bN; p_1, q_1, \dots, p_n, q_n \in P; p_1^{-1} q_1\cdots p_n^{-1}q_n = g\}.
\]
The coaction by $G$ is implemented by the unitary operator
\[
U \colon \ell^2(P) \otimes \ell^2(G) \to \ell^2(P) \otimes \ell^2(G); U(\de_s \otimes \de_g) = \de_s \otimes \de_{sg} \foral s \in P, g \in G.
\]
A routine calculation shows that the $*$-homomorphism
\[
\ca_\la(P) \stackrel{\simeq}{\longrightarrow}
\ca(V_p \otimes I \mid p \in P) \stackrel{\ad_{U}}{\longrightarrow}
\ca(V_p \otimes \la_p \mid p \in P)
\]
is a reduced coaction, and thus it lifts to a normal coaction on $\ca_\la(P)$.
The induced faithful conditional expectation on $\ca_\la(P)$ coincides with the compression to the diagonal $\ell^\infty(P)$.
In the presence of an overlying group there is an explicit formula for the projection $E_{[\Bx]}$ for $\Bx \in \J$, obtained in \cite[Lemma 3.1]{Li12}, i.e.,
\[
E_{[\Bx]} = V_{p_1}^* V_{q_1} \cdots V_{p_n}^* V_{q_n},
\]
for any $p_1,q_1, \dots, p_n, q_n \in P$ satisfying $\Bx=q_n^{-1} p_n \dots q_1^{-1} p_1 P$ and  $p_1^{-1} q_1 \cdots p_n^{-1} q_n = e$ in $G$.

\subsection{C*-correspondences}

The theory of Hilbert modules over C*-algebras is well-developed.
The reader is addressed to \cite{Lan95,MT05} for an excellent introduction to the subject; see also \cite{Kat03} and the references therein.

A \emph{C*-correspondence} $X$ over $A$ is a right Hilbert module over $A$ with a left action given by a $*$-homomorphism $\vphi_X \colon A \to \L(X)$, where $\L(X)$ denotes the C*-algebra of adjointable operators on $X$.
A C*-correspondence $X$ over $A$ is called \emph{non-degenerate} if $[\vphi_X(A) X] = X$.
We write $\K(X)$ for the closed linear span of the rank one adjointable operators $\theta_{\xi, \eta}$.
For two C*-corresponden\-ces $X, Y$ over the same $A$ we write $X \otimes_A Y$ for the balanced tensor product over $A$.
We say that $X$ is \emph{unitarily equivalent} to $Y$ (symb. $X \simeq Y$) if there is a surjective adjointable operator $U \in \L(X,Y)$ that is an $A$-bimodule map and $\sca{U \xi, U \eta} = \sca{\xi, \eta}$ for all $\xi, \eta \in X$.

A representation of a C*-correspondence $X$ over $A$ is a pair $(t_0,t_1)$ where $t_0\colon A\to \B(H)$ is a $*$-homomorphism and $t_1\colon X\to \B(H)$ is a linear map that satisfies the following:
\begin{enumerate}
    \item $t_0(a)t_1(\xi)=t_1(\varphi_X(a)\xi)$ for all $\xi\in X$ and $a\in A$,
    \item $t_1(\xi)^* t_1(\eta) = t_0(\sca{\xi,\eta})$ for all $\xi, \eta \in X$.
\end{enumerate} 
We note that condition (ii) also implies that $t_1(\xi)t_0(a)=t_1(\xi a)$.  
Every representation $(t_0,t_1)$ as above defines a $*$-representation 
\[
t^{(1)}\colon \K(X)\to \B(H); \theta_{\xi, \eta} \mapsto t(\xi) t(\eta)^*
\foral \xi, \eta \in X.
\] 
If $t_0$ is injective, then both $t_1$ and $t^{(1)}$ are isometric.

\subsection{Product systems}

In his seminal paper, Fowler \cite{Fow02} formalised product systems motivated by the work of Nica \cite{Nic92}, Pimsner \cite{Pim95}, as well as his work with Raeburn \cite{FR98}.
However this definition imposes further conditions on the family of C*-correspondences.
In this subsection we consider a more general framework that still allows the structural constructions relevant to product systems to be formulated yet without these extra conditions.
In Subsection \ref{Ss:Fow} we will give the connection with, and the comparison to, the product systems in the sense of Fowler.

Let $P$ be a unital left-cancellative discrete semigroup.
We say that a family $X = \{X_p\}_{p \in P}$ of closed operator spaces in a common $\B(H)$ is a \emph{(concrete) product system} if the following are satisfied:
\begin{enumerate}
\item $A := X_{e}$ is a C*-algebra;
\item $X_p \cdot X_q \subseteq X_{pq}$ for all $p, q \in P$;
\item $X_p^* \cdot X_{pq} \subseteq X_q$ for all $p, q \in P$.
\end{enumerate}
Uniqueness of $q \in P$ in item (iii) follows since $P$ is left-cancellative.
Moreover it follows that each $X_p$ is a C*-correspondence over $A$.

The properties of a concrete product system are enough to provide a Fock space representation.
Towards this end, consider $X$ in some $\B(H)$, and let $\B(H)$ with its trivial C*-correspondence structure.
We will be writing 
\[
\sca{\cdot, \cdot}_{p} \colon X_p \times X_p \longrightarrow A ; (\xi_p, \eta_p) \mapsto \sca{\xi_p, \eta_p}_p := \xi_p^* \cdot \eta_p,
\]
for the inner product induced on the $X_p$.
For every $\xi_p \in X_p$ we define the multiplication operator
\[
M_{\xi_p}^{q, pq} \colon X_q \longrightarrow X_{pq} ; \eta_q \mapsto \xi_p \cdot \eta_q,
\]
with the multiplication taking place inside $\B(H)$, satisfying
\[
\|M_{\xi_p}^{q, pq}\|_{\B(X_q, X_{pq})} \leq \nor{\xi_p}_{X_p}.
\]
Associativity of the product yields $M_{\xi_p}^{q, pq} \in \L(X_q, X_{pq})$ with
\[
(M_{\xi_p}^{q, pq} )^* \colon X_{pq} \to X_q; 
\eta_{pq}\mapsto \xi_p^* \cdot \eta_{pq} \in X_q.
\]
Consider the Fock space $\F X := \sum^\oplus_{r \in P} X_r$, as a right Hilbert module over $A$.
For every $\xi_p \in X_p$, define the \emph{left creation operator}
\[
\la_p(\xi_p) := \sumoplus_{r \in P} M_{\xi_p}^{r, pr} \; \text{ so that } \; \la_p(\xi_p)^* = \sumoplus_{r \in P} (M_{\xi_p}^{r, pr})^*,
\]
where the sum is taken in the s*-topology, and with the understanding that we are embedding $\L(X_r, X_{pr}) \hookrightarrow \L(\F X)$ as the $(pr,r)$-entry.
By applying on $\eta_r \in X_r \subseteq \F X$ we derive that
\[
\la_p(\xi_p) \eta_r = \xi_p \cdot \eta_r
\qand
\la_p(\xi_p)^* \eta_r =
\begin{cases}
\xi_p^* \cdot \eta_r & \text{if } r \in pP, \\
0 & \text{if } r \notin pP.
\end{cases}
\]
We write $\la \colon X \to \L(\F X)$ for this map, and we refer to $\la=\{\la_p\}_{p\in P}$ as the \emph{Fock representation of $X$}.
We will write $\T_\la(X)$ for the \emph{Fock C*-algebra} defined as $\ca(\la_p(X_p) \mid p \in P)$.

More generally, a \emph{(Toeplitz) representation $t = \{t_p\}_{p \in P}$ of $X$} consists of a family of linear maps $t_p$ of $X_p$ such that:
\begin{enumerate}
\item $t_e$ is a $*$-representation of $A := X_e$;
\item $t_p(\xi_p) t_q(\xi_q) = t_{pq}(\xi_p \cdot \xi_q)$ for all $\xi_p \in X_p$ and $\xi_q \in X_q$;
\item $t_p(\xi_p)^* t_{pq}(\xi_{pq}) = t_q( \xi_p^* \cdot \xi_{pq})$ for all $\xi_p \in X_p$ and $\xi_{pq} \in X_{pq}$.
\end{enumerate}
By definition every pair $(t_e,t_p)$ is a representation of the C*-correspondence $X_p$ over $A$.
A representation $t$ is called \emph{injective} if $t_{e}$ is injective on $A$.
It transpires that the Fock representation is an injective representation of $X$.
The \emph{Toeplitz algebra $\T(X)$ of $X$} is the universal C*-algebra generated by $X=\{X_p\}_{p\in P}$ with respect to the representations of $X$.
If $t$ defines a representation of $X$, then we will write $t_\ast \colon \T(X) \to \B(H)$ for the induced $*$-representation of $\T(X)$ and
\[
\ca(t) \equiv \ca(t_\ast) := t_\ast(\T(X)) = \ca(t_p(X_p) \mid p \in P).
\]

Henceforth we will assume that $P$ is a unital subsemigroup of a group $G$.

\begin{proposition} \label{P:mult}
Let $P$ be a unital subsemigroup of a discrete group $G$ and let $X$ be a product system over $P$.
Let $t$ be a representation of $X$ and let $\xi_r \in X_r$, $\xi_{p_i} \in X_{p_i}$, and $\xi_{q_i} \in X_{q_i}$ for $r, p_1, q_1, \dots, p_n, q_n \in P$ and $\eps, \eps' \in \{0,1\}$.
If $r \in q_n^{-\eps'} p_n \dots q_1^{-1} p_1^{\eps} P$, then $p_1^{-\eps} q_1 \cdots p_n^{-1}q_n^{\eps'} r \in P$ and
\begin{equation*}
\left( t_{p_1}(\xi_{p_1})^* \right)^{\eps} t_{q_1}(\xi_{q_1}) \cdots t_{p_n}(\xi_{p_n})^* t_{q_n} (\xi_{q_n})^{\eps'} t_r(\xi_r)
=
t_{p_1^{-\eps} q_1 \cdots p_n^{-1}q_n^{\eps'} r} \left( (\xi_{p_1}^*)^{\eps} \xi_{q_1} \cdots \xi_{p_n}^* \xi_{q_n}^{\eps'} \xi_r \right).
\end{equation*}
\end{proposition}

\begin{proof}
We start with the following three comments. 
First, suppose that $s=p^{-1}r$ for some $p, s \in P$.
Then by definition we have
\[
t_p(\xi_p)^*t_r(\xi_r)
=
t_p(\xi_p)^*t_{ps}(\xi_r)
=
t_{s}(\xi_p^*\xi_r).
\]
Next, suppose that $s=qr$ for some $q,s\in P$.
Then by definition we have
\[
t_q(\xi_q)t_r(\xi_r)
=
t_{qr}(\xi_q\xi_r)
=
t_s(\xi_q\xi_r).
\]
Finally, suppose that $s=p^{-1}qr$ for some $p,q,s\in P$.
Then we have 
\[
t_p(\xi_p)^*t_q(\xi_q)t_r(\xi_r)=
t_p(\xi_p)^*t_{qr}(\xi_q\xi_r)=
t_p(\xi_p)^*t_{ps}(\xi_q\xi_r)=
t_s(\xi_p^*\xi_q\xi_r).
\]

We proceed with the proof. 
Let $r$ be in $q_n^{-\eps'} p_n \dots q_1^{-1} p_1^{\eps} P$ for some $\eps, \eps' \in \{0,1\}$. 
We have that there is an $s_1 \in q_{n-1}^{-1} p_{n-1} \dots q_1^{-1} p_1^{\eps} P$ such that $s_1 = p_n^{-1} q_n^{\eps'} r$.
From the comments above we then have $\xi_{p_n}^* \xi_{q_n}^{\eps'} \xi_r \in X_{s_1}$ and
\[
t_{p_n}(\xi_{p_n})^* t_{q_n}(\xi_{q_n})^{\eps'} t_r(\xi_r)
=
t_{s_1}(\xi_{p_n}^* \xi_{q_n}^{\eps'} \xi_r).
\]
Since $s_1\in q_{n-1}^{-1} p_{n-1} \dots q_1^{-1} p_1^{\eps} P$  we have that there is an $s_2\in q_{n-2}^{-1} p_{n-2} \dots q_1^{-1} p_1^{\eps} P$ such that $s_2=p_{n-1}^{-1}q_{n-1}s_1$, and therefore 
\[
t_{p_{n-1}}(\xi_{p_{n-1}})^* t_{q_{n-1}}(\xi_{q_{n-1}}) t_{s_1}(\xi_{p_n}^*\xi_{q_n}^{\eps'}\xi_r)
=
t_{s_2}(\xi_{p_{n-1}}^*\xi_{q_{n-1}}\xi_{p_n}^* \xi_{q_n}^{\eps'} \xi_r).
\]
Continuing inductively for each $k=3,\dots,n-1$ we obtain an $s_k\in q_{n-k}^{-1} p_{n-k} \dots q_1^{-1} p_1^{\eps} P$ such that $s_k=p_{n-(k-1)}^{-1}q_{n-(k-1)}s_{k-1}$, and for each $k=1,\dots,n-1$ we have
\begin{align*}
&
t_{p_{n-(k-1)}}(\xi_{p_{n-(k-1)}})^* t_{q_{n-(k-1)}}(\xi_{q_{n-(k-1)}}) \cdots t_{p_n}(\xi_{p_n})^* t_{q_n}(\xi_{q_n}) ^{\eps'}t_r(\xi_r)= \\
&\hspace{5cm} =
t_{s_k}(\xi_{p_{n-(k-1)}}^* \xi_{q_{n-(k-1)}} \cdots \xi_{p_n}^* \xi_{q_n}^{\eps'} \xi_r)\\
&\hspace{5cm} =
t_{p_{n-(k-1)}^{-1} q_{n-(k-1)} \cdots p_n^{-1} q_n^{\eps'} r}(\xi_{p_{n-(k-1)}}^* \xi_{q_{n-(k-1)}} \cdots \xi_{p_n}^* \xi_{q_n}^{\eps'} \xi_r).
\end{align*}
Note that $r=q_n^{-\eps'} p_n \cdots q_1^{-1} p_1^{\eps} s_n$ for some $s_n \in P$, and thus $p_1^{-\eps} q_1 \cdots p_n^{-1}q_n^{\eps'} r= s_n \in P$.
Since $p_2^{-1} q_2 \cdots p_n^{-1} q_n^{\eps'} r = q_1^{-1} p_1^{\eps} s_n$, from the comments above we obtain
\begin{align*}
&(t_{p_1}(\xi_{p_1})^*)^{\eps} t_{q_1}(\xi_{q_1})t_{p_2}(\xi_{p_2})^* t_{q_2}(\xi_{q_2})\cdots t_{p_n}(\xi_{p_n})^* t_{q_n}(\xi_{q_n}) ^{\eps'}t_r(\xi_r) =\\
& \hspace{5cm} =
(t_{p_1}(\xi_{p_1})^*)^{\eps}t_{q_1}(\xi_{q_1})t_{p_2^{-1} q_2 \cdots p_n^{-1} q_n^{\eps'} r}(\xi_{p_2}^* \xi_{q_2} \cdots \xi_{p_n}^* \xi_{q_n}^{\eps'} \xi_r)\\
& \hspace{5cm} = 
(t_{p_1}(\xi_{p_1})^*)^{\eps}t_{q_1}(\xi_{q_1})t_{q_1^{-1}p_1^{\eps} s_n}(\xi_{p_2}^* \xi_{q_2} \cdots \xi_{p_n}^* \xi_{q_n}^{\eps'} \xi_r)\\
& \hspace{5cm} =
t_{s_n}((\xi_{p_1}^*)^{\eps} \xi_{q_1} \cdots \xi_{p_n}^* \xi_{q_n}^{\eps'} \xi_r),
\end{align*}
as required.
\end{proof}

For the Fock representation we have the following proposition, see also \cite[page 340]{Fow02}.

\begin{proposition}\label{P:fc mt}
Let $P$ be a unital subsemigroup of a discrete group $G$ and let $X$ be a product system over $P$. 
Let  $p_1, q_1, \dots, p_n, q_n \in P$ and $\eps, \eps' \in \{0,1\}$, and $\xi_{p_i}\in X_{p_i}$ and $\xi_{q_i}\in X_{q_i}$ for $i=1,\dots,n$. Then for each $r\in P$ and $\xi_r\in X_r$ we have
\[
\left( \la_{p_1}(\xi_{p_1})^* \right)^\eps \la_{q_1}(\xi_{q_1}) \cdots \la_{p_n}(\xi_{p_n})^* \la_{q_n}(\xi_{q_n})^{\eps'}\xi_r =
\begin{cases}
(\xi_{p_1}^*)^{\eps}\xi_{q_1}\cdots \xi_{p_n}^*\xi_{q_n}^{\eps'}\xi_r & \text{if } r\in q_n^{-\eps'} p_n \dots q_1^{-1} p_1^{\eps} P, \\
0 & \text{if } r \not\in q_n^{-\eps'} p_n \dots q_1^{-1} p_1^{\eps} P.
\end{cases}
\]
\end{proposition}

\begin{proof}
If $r \in q_n^{-\eps'} p_n \dots q_1^{-1} p_1^{\eps} P$ and $a\in A:=X_e$, then by Proposition \ref{P:mult} we have
\begin{align*}
&\left(\la_{p_1}(\xi_{p_1})^*\right)^{\eps} \la_{q_1}(\xi_{q_1}) \cdots \la_{p_n}(\xi_{p_n})^* \la_{q_n} (\xi_{q_n})^{\eps'}\xi_r a=\\
& \hspace{4cm} =
\left(\la_{p_1}(\xi_{p_1})^*\right)^{\eps} \la_{q_1}(\xi_{q_1}) \cdots \la_{p_n}(\xi_{p_n})^* \la_{q_n} (\xi_{q_n})^{\eps'} \la_r(\xi_r) a\\
& \hspace{4cm} =
\la_{p_1^{-\eps} q_1\cdots p_n^{-1}q_n^{\eps'} r}((\xi_{p_1}^*)^{\eps} \xi_{q_1} \cdots \xi_{p_n}^* \xi_{q_n}^{\eps'} \xi_r)a\\
& \hspace{4cm} =
(\xi_{p_1}^*)^{\eps} \xi_{q_1} \cdots \xi_{p_n}^* \xi_{q_n}^{\eps'} \xi_r a.
\end{align*}
Since $[X_r A] = X_r$ we get 
\[
(\la_{p_1}(\xi_{p_1})^*)^{\eps} \la_{q_1}(\xi_{q_1}) \cdots \la_{p_n}(\xi_{p_n})^* \la_{q_n} (\xi_{q_n})^{\eps'}\xi_r
=
(\xi_{p_1}^*)^{\eps} \xi_{q_1} \cdots \xi_{p_n}^* \xi_{q_n}^{\eps'} \xi_r.
\]

Next let $r\not\in q_n^{-\eps'} p_n \dots q_1^{-1} p_1^{\eps} P$, and suppose towards a contradiction that
\[
(\la_{p_1}(\xi_{p_1})^*)^{\eps} \la_{q_1}(\xi_{q_1}) \cdots \la_{p_n}(\xi_{p_n})^* \la_{q_n}(\xi_{q_n})^{\eps'} \xi_r\neq 0.
\]
In particular we have
\[
\la_{p_n}(\xi_{p_n})^* \xi_{q_n}^{\eps'}\xi_r
=
\la_{p_n}(\xi_{p_n})^* \la_{q_n}(\xi_{q_n})^{\eps'} \xi_r
\neq 
0.
\]
Hence $p_nr_1=q_n^{\eps'}r$ for some $r_1\in P$ and 
\[
0 \neq \la_{p_n}(\xi_{p_n})^* \xi_{q_n}^{\eps'}\xi_r=\xi_{p_n}^*\xi_{q_n}^{\eps'}\xi_r\in X_{r_1}.
\]
Inductively, for $k=1,\dots, n$ we obtain $r_k\in P$ such that 
\begin{align*}
p_n r_1 =q_n^{\eps'}r, \,
p_1^{\eps}r_n = q_1r_{n-1},
\text{ and }
p_{n-(k-1)} r_k = q_{n-(k-1)} r_{k-1} 
\foral k=2, \dots, n-1,
\end{align*}
and therefore $r=q_n^{-\eps'}p_n\cdots q_{n-(k-1)}^{-1}p_{n-(k-1)}r_k$ for each $k=2,\dots,n$. Thus,
\[
r
\in 
P \cap (q_n^{-\eps'} \cdot p_n \cdot P) \cap (q_n^{-\eps'} \cdot p_n \cdot q_{n-1}^{-1} \cdot p_{n-1} \cdot P) \cap \cdots \cap (q_n^{-\eps'} \cdot p_n \cdots q_1^{-1} \cdot p_1^{\eps} \cdot P)
=
q_n^{-\eps'} p_n \dots q_1^{-1} p_1^{\eps} P,
\]
which is a contradiction.
\end{proof}

A representation $t = \{t_p\}_{p \in P}$ of $X$ will be called \emph{equivariant} if there exists a $*$-homomorphism $\de$ of $\ca(t)$ such that
\[
\de \colon \ca(t) \to \ca(t) \otimes \ca_{\max}(G) ; t_p(\xi_p) \mapsto t_p(\xi_p) \otimes u_{p}.
\]
It follows that $\de$ is injective with a left inverse given by the map $\id \otimes \chi$.
Moreover, it satisfies the coaction identity and hence $\ca(t)$ admits a coaction by $G$.
For simplicity, we will say that \emph{$t$ admits a coaction by $G$} if such a $\de$ exists. 
In this case the $g$-fiber $[\ca(t)]_g$, for $g \in G$, is the closed linear span of the elements
\begin{align*}
\left( t_{p_1}(X_{p_1})^* \right)^{\eps} t_{q_1}(X_{q_1}) \cdots t_{p_n}(X_{p_n})^*t_{q_n} (X_{q_n})^{\eps'}
\text{ such that } 
p_1^{-\eps} q_1 \cdots p_n^{-1} q_n^{\eps'} = g,
\end{align*}
for $\eps, \eps' \in \{0,1\}$ and $n \in \bN$. 
A proof can be found in \cite[Lemma 2.2]{Seh18} for non-degenerate product systems, but similar arguments give the conclusion in the general case.

Let $t$ be a representation of $X$.
For $\Bx \in \J$ we define the \emph{$\bo{K}$-core on $\Bx$} of $\ca(t)$ to be the closed linear span of the spaces
\begin{align*}
\left( t_{p_1}(X_{p_1})^* \right)^{\eps} t_{q_1}(X_{q_1}) \cdots t_{p_n}(X_{p_n})^* 
t_{q_n} (X_{q_n})^{\eps'}
\end{align*}
for any $p_1, q_1, \dots, p_n, q_n \in P$ and $\eps, \eps' \in \{0,1\}$ that satisfy
\begin{align*}
p_1^{-\eps} q_1 \cdots p_n^{-1} q_n^{\eps'} = e
\text{ and }
q_n^{-\eps'} p_n \dots q_1^{-1} p_1^{\eps} P = \Bx.
\end{align*}
We write $\bo{K}_{\Bx, t_\ast}$ for this closed linear space.
We do not claim that $\bo{K}_{\mt, t_{\ast}} = (0)$.
Note that, if $p_1^{-\eps} q_1 \cdots p_n^{-1} q_n^{\eps'} = e$, then
\[
q_n^{-\eps'} p_n \dots q_1^{-1} p_1^{\eps} P 
=
p_1^{-\eps} q_1 \dots p_n^{-1} q_n^{\eps'} P.
\]
Indeed, since $p_1^{-\eps} q_1 \cdots p_n^{-1} q_n^{\eps'} = e$ we have
\begin{align*}
q_n^{-\eps'} p_n \dots q_1^{-1} p_1^{\eps} P 
& =
P \cap (q_n^{-\eps'} \cdot p_n \cdot P) 
\cap \cdots \cap (q_n^{-\eps'} \cdot p_n \cdots q_1^{-1} \cdot p_1^{\eps} \cdot P) \\
& =
(p_1^{-\eps}\cdot q_1 \cdots p_n^{-1} \cdot q_n^{\eps'}\cdot P)\cap (p_1^{-\eps}\cdot q_1\cdots p_{n-1}^{-1}\cdot q_{n-1}\cdot P)\cap \cdots 
\cap P \\
& =
p_1^{-\eps} q_1 \dots p_n^{-1} q_n^{\eps'} P.
\end{align*}
We conclude that each $\bo{K}_{\Bx, t_\ast}$ is a selfadjoint space.
Moreover, if we have
\[
\Bx = q_n^{-\eps_1'} p_n \dots q_1^{-1} p_1^{\eps_1} P \qand \By = s_m^{-\eps_2'} r_m \dots s_1^{-1} r_1^{\eps_2} P,
\]
with $p_1^{-\eps_1} q_1 \cdots p_n^{-1} q_n^{\eps_1'} = e$ and $r_1^{-\eps_2} s_1 \cdots r_m^{-1} s_m^{\eps_2'} = e$, then
\begin{align*}
\Bx \cap \By 
& =
\left(P \cap (q_n^{-\eps_1'} \cdot p_n \cdot P) 
\cap \cdots \cap (q_n^{-\eps_1'} \cdot p_n \cdots q_1^{-1} \cdot p_1^{\eps_1} \cdot P)\right)\bigcap
\\
& \hspace{3cm}
\bigcap \left( P \cap (s_m^{-\eps_2'} \cdot r_m \cdot P) 
\cap \cdots \cap (s_m^{-\eps_2'} \cdot r_m \cdots s_1^{-1} \cdot r_1^{\eps_2} \cdot P)\right)\\
&=
P \cap (q_n^{-\eps_1'} \cdot p_n \cdot P) \cap \cdots \cap (q_n^{-\eps_1'} \cdot p_n \cdots q_1^{-1} \cdot p_1^{\eps_1} \cdot P)
\bigcap \\
& \hspace{3cm}
\bigcap (q_n^{-\eps_1'} \cdot p_n \cdots q_1^{-1} \cdot p_1^{\eps_1} \cdot s_m^{-\eps_2'}\cdot r_m \cdot P) \cap \cdots \\
& \hspace{4cm} \cdots
\cap (q_n^{-\eps_1'} \cdot p_n \cdots q_1^{-1} \cdot p_1^{\eps_1} \cdot s_m^{-\eps_2'}\cdot r_m\cdots s_1^{-1}\cdot r_1^{\eps_2}\cdot P)\\
&=
q_n^{-\eps_1'} p_n \dots q_1^{-1} p_1^{\eps_1} s_m^{-\eps_2'} r_m \dots s_1^{-1} r_1^{\eps_2} P,
\end{align*}
with $r_1^{-\eps_2} s_1 \cdots r_m^{-1} s_m^{\eps_2'} p_1^{-\eps_1} q_1 \cdots p_n^{-1} q_n^{\eps_1'} = e$.
Hence we obtain
\[
\bo{K}_{\Bx, t_\ast} \cdot \bo{K}_{\By, t_\ast} \subseteq \bo{K}_{\Bx \cap \By, t_\ast}.
\]
From this we derive that every $\bo{K}_{\Bx, t_\ast}$ is an algebra (as $\Bx \cap \Bx = \Bx$), and thus a C*-subalgebra of $\ca(t)$.
More generally, for a finite $\cap$-closed $\F \subseteq \J$ we define the \emph{$\bo{B}$-core on $\F$} by
\[
\bo{B}_{\F, t_\ast} := \sum_{\Bx \in \F} \bo{K}_{\Bx, t_\ast}.
\]
By the discussion above, it follows that if $\F$ is $\cap$-closed and $\Bx \in \J$ is such that $\Bx \cap \By \in \F$ for every $\By \in \F$, then $\bo{B}_{\F, t_\ast}$ is an ideal in $\bo{B}_{\F \cup \{\Bx\}, t_\ast}$.
An induction argument gives that $\bo{B}_{\F, t_\ast}$ is a C*-subalgebra of $\ca(t)$ when $\F \subseteq \J$ is finite and $\cap$-closed.
In particular we have that, if $t$ is equivariant, then the fixed point algebra $[\ca(t)]_e$ is the inductive limit of the C*-subalgebras $\bo{B}_{\F, t_\ast}$ over $\cap$-closed and finite $\F \subseteq \J$ (with respect to inclusion of sets).

\begin{proposition} \label{P:taut}
Let $P$ be a unital subsemigroup of a discrete group $G$ and let $X$ be a product system over $P$.
Let $t$ be a representation of $X$ and let $\hat{t}$ be the representation of $X$ such that $\T(X) = \ca(\hat{t})$.
For $\mt \neq \Bx \in \J$ and $r \in \Bx$ we have
\[
t_\ast(b_{\Bx}) t_r(\xi_r) = t_r( \la_{\ast}(b_{\Bx}) \xi_r)
\foral b_{\Bx} \in \bo{K}_{\Bx, \hat{t}_\ast}, \xi_r \in X_r.
\]
\end{proposition}

\begin{proof}
First consider an element of the form
\[
b_{\bo{x}} 
:= 
\left({\hat{t}_{p_1}(\xi_{p_1})}^*\right)^{\eps} {\hat{t}_{q_1}(\xi_{q_1})}\cdots{\hat{t}_{p_n}(\xi_{p_n})}^* \hat{t}_{q_n}(\xi_{q_n})^{\eps'}
\in 
\bo{K}_{\bo{x},\hat{t}_\ast},
\]
such that $\Bx = q_n^{-\eps'} p_n \dots q_1^{-1} p_1^{\eps} P$ and $p_1^{-\eps}q_1 \cdots p_n^{-1}q_n^{\eps'} = e$ for some $p_1, q_1, \dots, p_n, q_n \in P$ and $\eps, \eps' \in \{0,1\}$.
Let $r \in \Bx$.
By Proposition \ref{P:mult} and Proposition \ref{P:fc mt} we have
\begin{align*}
t_\ast(b_{\Bx}) t_r(\xi_r)
& =
t_\ast \left( \left({\hat{t}_{p_1}(\xi_{p_1})}^*\right)^{\eps} {\hat{t}_{q_1}(\xi_{q_1})}\cdots{\hat{t}_{p_n}(\xi_{p_n})}^* \hat{t}_{q_n}(\xi_{q_n})^{\eps'} \right) t_r(\xi_r) \\
& =
\left( t_{p_1}(\xi_{p_1})^* \right)^{\eps} t_{q_1}(\xi_{q_1}) \cdots t_{p_n}(\xi_{p_n})^* t_{q_n} (\xi_{q_n})^{\eps'} t_r(\xi_r)\\
& =
t_r\left((\xi_{p_1}^*)^{\eps} \xi_{q_1} \cdots \xi_{p_n}^* \xi_{q_n}^{\eps'} \xi_r\right)
=
t_r( \la_\ast(b_{\bo{x}})\xi_r).
\end{align*}
Taking finite linear combinations and their norm-limits completes the proof.
\end{proof}

The following proposition is an immediate consequence of Proposition \ref{P:fc mt}.

\begin{proposition}\label{P:fc mt 2}
Let $P$ be a unital subsemigroup of a discrete group $G$ and let $X$ be a product system over $P$.
Let $\hat{t}$ be the representation of $X$ such that $\T(X) = \ca(\hat{t})$. Then for every $\Bx \in \J$ and $r \not\in \Bx$ we have
\[
\la_{\ast}(b_{\bo{x}})\xi_r=0
\foral b_{\bo{x}}\in \bo{K}_{\bo{x},\hat{t}_\ast}, \xi_r\in X_r.
\]
\end{proposition}
\begin{proof}
We consider an element of the form
\[
b_{\bo{x}} 
:= 
\left({\hat{t}_{p_1}(\xi_{p_1})}^*\right)^{\eps} {\hat{t}_{q_1}(\xi_{q_1})}\cdots{\hat{t}_{p_n}(\xi_{p_n})}^* \hat{t}_{q_n}(\xi_{q_n})^{\eps'}
\in 
\bo{K}_{\bo{x},\hat{t}_\ast},
\]
such that $\Bx = q_n^{-\eps'} p_n \dots q_1^{-1} p_1^{\eps} P$ and $p_1^{-\eps}q_1 \cdots p_n^{-1}q_n^{\eps'} = e$ for some $p_1, q_1, \dots, p_n, q_n \in P$ and $\eps, \eps' \in \{0,1\}$.
Let $\xi_r\in X_r$. 
By Proposition \ref{P:fc mt} we obtain
\[
\la_\ast(b_{\bo{x}})\xi_r
=
\left( \la_{p_1}(\xi_{p_1})^* \right)^\eps \la_{q_1}(\xi_{q_1}) \cdots \la_{p_n}(\xi_{p_n})^* \la_{q_n}(\xi_{q_n})^{\eps'}\xi_r
=
0.
\]
Taking finite linear combinations and their norm-limits completes the proof.
\end{proof}

By the universal property, the Toeplitz C*-algebra $\T(X)$ admits a coaction by $G$, see for example \cite[Lemma 2.2]{Seh18}.
We let 
\[
\P\S X:=\Big\{[\T(X)]_g\Big\}_{g\in G}
\]
be the induced Fell bundle. 
Moreover, the embedding $X \hookrightarrow \T(X)$ is injective, since the Fock representation $\la$ is injective.
Therefore for each $p \in P$ we have an embedding $X_p \hookrightarrow [\P\S X]_p$.

\begin{proposition} \label{P:max psx}
Let $P$ be a unital subsemigroup of a discrete group $G$ and let $X$ be a product system over $P$. 
Then the family $\{X_p \hookrightarrow [\P\S X]_p\}_{p \in P}$ of embeddings lifts to a $*$-isomorphism 
\[
\T(X) \simeq \ca_{\max}(\P\S X).
\]
\end{proposition}

\begin{proof}
Let $\hat{t}$ be the representation of $X$ such that $\ca(\hat{t})=\T(X)$ and let $\hat{j}\colon \P\S X\to \ca_{\max}(\P\S X)$ be the induced canonical embedding of the Fell bundle. 
Let $t=\{ t_p\}_{p\in P}$ be the family of maps defined by $t_p\colon X_p \to \ca_{\max}(\P\S X)$ such that $t_p(\xi_p)=\hat{j}_p(\hat{t}_p(\xi_p))$ for each $\xi_p\in X_p$ and $p\in P$. 
Then $t$ is a representation of $X$.
Indeed, for the first axiom, since $\hat{t}_e$ and $\hat{j}_e$ are $*$-representations we obtain that $t_e$ is a $*$-representation of $A:=X_e$. 
For the second axiom, let $\xi_p\in X_p$ and $\xi_q\in X_q$, then 
\[
t_p(\xi_p)t_q(\xi_q)=\hat{j}_p(\hat{t}_p(\xi_p))\hat{j}_q(\hat{t}_q(\xi_q))=\hat{j}_{pq}(\hat{t}_{pq}(\xi_p\xi_q))=t_{pq}(\xi_p\xi_q).
\]
For the third axiom, let $\xi_p\in X_p$ and $\xi_{pq}\in X_{pq}$, then 
\[
t_p(\xi_p)^*t_{pq}(\xi_{pq})=\hat{j}_p(\hat{t}_p(\xi_p))^*\hat{j}_{pq}(\hat{t}_{pq}(\xi_{pq}))=\hat{j}_{p^{-1}}(\hat{t}_p(\xi_p)^*)\hat{j}_{pq}(\hat{t}_{pq}(\xi_{pq}))=\hat{j}_q(\hat{t}_q(\xi_p^*\xi_{pq}))=
t_q(\xi_p^*\xi_{pq}),
\]
as required.

Therefore, there is an induced $*$-epimorphism
\[
\T(X) \to\ca(t)=\ca_{\max}(\P\S X); \hat{t}_p(\xi_p) \mapsto t_p(\xi_p)=\hat{j}_p(\hat{t}_p(\xi_{p})).
\]
On the other hand, by \cite[Theorem 19.5]{Exe17} we obtain a $*$-epimorphism
\[
\ca_{\max}(\P\S X) \to \T(X); \hat{j}_p(\hat{t}_p(\xi_p)) \mapsto \hat{t}_p(\xi_{p})
\]
in the other direction, and the proof is complete.
\end{proof}

As noted in \cite[Proposition 4.1]{DKKLL20}, the representation $\la$ admits a reduced coaction by using the unitary
\[
U \colon \F X \otimes \ell^2(G) \to \F X \otimes \ell^2(G); U(\xi_r \otimes \de_g) = \xi_r \otimes \de_{rg}.
\]
Indeed, for this $U$ we have
\[
U \cdot (\la_p(\xi_p) \otimes I) = (\la_p(\xi_p) \otimes \la_p) \cdot U
\foral
p \in P,
\]
and therefore the map
\[
\T_\la(X) \stackrel{\simeq}{\longrightarrow}
\ca(\la_p(\xi_p) \otimes I \mid p \in P) \stackrel{\ad_{U}}{\longrightarrow}
\ca(\la_p(\xi_p) \otimes \la_p \mid p \in P)
\]
defines a reduced coaction on $\T_\la(X)$.
Consequently, it lifts to a normal coaction $\de$ on $\T_\la(X)$.

\subsection{Fowler's product systems} \label{Ss:Fow}

In \cite[Definition 2.1]{Fow02} Fowler defines a product system $X$ over $P$ with coefficients in a C*-algebra $A$ as a family  $\{X_p\}_{p \in P}$ of C*-correspondences over $A$ together with a family of maps $\{u_{p,q}\colon X_p\otimes_A X_q\to X_{pq}\}_{p,q \in P}$ such that:
\begin{enumerate}
\item the space $X_e$ is the C*-correspondence $A$ over $A$ where the left and right actions of $A$ is multiplication on $A$;
\item if $p, q \neq e$, then $u_{p,q}$ is a unitary; 
\item if $p=e$, then $u_{e,q}\colon A\otimes_A X_q\to X_q$ is given by the left action of $A$ on $X_q$ for $q \in P$;
\item if $q=e$, then $u_{p,e}\colon X_p\otimes_A A\to X_p$ is given by the right action of $A$ on $X_p$ for $p \in P$;
\item the maps $\{u_{p,q}\}_{p,q \in P}$ are associative in the sense that
\[
u_{p,qr} \circ (\id_{X_p} \otimes u_{q,r})
=
u_{pq,r} \circ (u_{p,q} \otimes \id_{X_r})
\foral
p,q,r \in P.
\]
\end{enumerate}
Brownlowe, Larsen and Stammeier note in \cite[Remark 6.2]{BLS18} 
that if the group of units of the semigroup is non-trivial, then every $X_p$ is automatically non-degenerate, see also \cite[Remark 1.3]{KL19b}.
Hence it has been accustomed to work within the non-degenerate setup.
To allow for comparisons, we will refer to $(X, \{u_{p,q}\}_{p,q \in P})$ as \emph{a product system in the sense of Fowler}.

It can be directly verified that, if $X$ is a concrete product system in some $\B(H)$ such that $[X_p \cdot X_q] = X_{pq}$ for all $p,q \in P$, then the family $\{X_p\}_{p \in P}$ defines a product system in the sense of Fowler by considering the unitary maps $\{u_{p,q}\}_{p,q \in P}$ given by
\[
u_{p,q} \colon X_p \otimes_A X_q \to X_{pq}; \xi_p \otimes \xi_q \mapsto \xi_{p} \cdot \xi_q.
\]

A \emph{(Toeplitz) representation} $t$ of a product system $X$ in the sense of Fowler consists of a family $\{t_p\}_{p\in P}$, where $(t_e,t_p)$ is a representation of the C*-correspondence $X_p$ for all $p\in P$, and
\[ 
t_p(\xi_p)t_q(\xi_q)=t_{pq}(u_{p,q}(\xi_p \otimes \xi_q)) \foral \xi_p\in X_p, \xi_q\in X_q, p, q\in P.
\]

If $(X, \{u_{p,q}^X\}_{p,q \in P})$ and $(Y, \{u_{p,q}^Y\}_{p,q \in P})$ are two product systems in the sense of Fowler with coefficients in two C*-algebras $A$ and $B$ respectively, then we say that $X$ is \emph{unitarily equivalent} to $Y$ if there is a family $\{W_p \colon X_p \to Y_p\}_{p \in P}$ of unitaries such that:
\begin{enumerate}
\item $W_e\colon A\to B$ is a $*$-isomorphism;
\item $\sca{W_p(\xi_p), W_p(\eta_p)}_{Y_p} = W_e(\sca{\xi_p,\eta_p}_{X_p})$ for all $\xi_p,\eta_p\in X_p$ and $p\in P\setminus\{e\}$;
\item $\vphi_{Y_p}(W_e(a)) W_p(\xi_p)=W_p(\vphi_{X_p}(a)\xi_p)$ for all $a\in A, \xi_p\in X_p$ and $p\in P\setminus\{e\}$;
\item $W_p(\xi_p) W_e(a) = W_p(\xi_pa)$ for all $a\in A, \xi_p\in X_p$ and $p\in P\setminus\{e\}$;
\item $u_{p,q}^Y \circ ( W_p\otimes W_q) = W_{pq} \circ u_{p,q}^X$ for all $p,q\in P$.
\end{enumerate}
In this case the representations of $X$ are in bijection with the representations of $Y$.
The reader may refer to \cite[Section 2.3]{DK23} for the full details.

Given a product system $X$ in the sense of Fowler, we can define the left-creation operators
\[
\la_p(\xi_p) := \sum_{r \in P} \tau_{r}^{pr}(\xi_p) \foral \xi_p \in X_p,
\]
as the s*-sum of the adjointable operators
\[
\tau_{r}^{pr}(\xi_p) \colon X_r \longrightarrow X_{pr} ; \eta_r \mapsto u_{p,r}(\xi_p \otimes \eta_r).
\]
It follows that the family $\la(X) := \{\la_p(X_p)\}_{p \in P}$ defines a concrete product system. 
Moreover, the family $\{\la_p\}_{p \in P}$ defines a unitary equivalence between $X$ and $\la(X)$.
Hence $X$ and $\la(X)$ admit the same representations.
In particular, if $\la'$ is the Fock representation of $\la(X)$, then $\ca(\la')$ is unitarily equivalent to $\ca(\la)$ by the unitary
\[
\bigoplus_{p \in P} \la_p \colon \F X \to \F \la(X); \xi_p \mapsto \la_p(\xi_p).
\]
Hence from an operator theoretic point of view, the concrete product systems that we will be using encompass Fowler's product systems.

\section{Fock covariant representations}\label{S:fock covariant}

\subsection{Fock covariant representations}\label{Ss:fock covariant}

Let $P$ be a unital subsemigroup of a discrete group $G$ and let $X$ be a product system over $P$. 
Recall that $\la$ is injective and admits a normal coaction by $G$.
Consider the induced ideal
\[
\J_{\cov}^{\fock} := \sca{\ker\la_\ast \cap [\T(X)]_e} \lhd \T(X),
\]
and write $\T_{\cov}^{\fock}(X)$ for the quotient of $\T(X)$ by $\J_{\cov}^{\fock}$.
Since $\J_{\cov}^{\fock}$ is an induced ideal of $\T(X)$, then $\T_{\cov}^{\fock}(X)$ inherits a coaction by $G$.
The \emph{Fock covariant} bundle of $X$ is the Fell bundle
\[
\F\C X := \Big\{ [\T_{\cov}^{\fock}(X)]_g \Big\}_{g \in G},
\]
defined by the coaction by $G$ on $\T_{\cov}^{\fock}(X)$.
A representation of $X$ that promotes to a representation of $\F\C X$ will be called a \emph{Fock covariant representation of $X$}.
We note that, the definition of the Fock covariant bundle does not depend on the choice of $G$.
This follows by applying the first part of the proof of \cite[Lemma 3.9]{Seh18} in this setting. Note also that since the bimodule properties are graded, every representation of $\F\C X$ is a representation of $X$.

Let $q_{\cov}^{\fock} \colon \T(X) \to \T_{\cov}^{\fock}(X)$ be the quotient map by the ideal $\J_{\cov}^{\fock}$.
Then 
\[
q_{\cov}^{\fock}(\P\S X)
= 
\Big\{ [\T_{\cov}^{\fock}(X)]_g \Big\}_{g \in G}
=
\F\C X,
\]
and thus combining Proposition \ref{P:qnt bnd} with Proposition \ref{P:max psx} we obtain a canonical $*$-isomorphism 
\[
\T_{\cov}^{\fock}(X)\simeq \ca_{\max}(\F\C X).
\]
By the definition of $\J_{\cov}^{\fock}$ we have
\[
\ker q_{\cov}^{\fock} \cap [\T(X)]_e
=
\J_{\cov}^{\fock} \cap [\T(X)]_e
=
\ker \la_\ast \cap [\T(X)]_e.
\]
By Proposition \ref{P:eq dgrm} we have
\[
\ker q_{\cov}^{\fock} \cap [\T(X)]_g
=
\ker \la_\ast \cap [\T(X)]_g
\foral
g \in G.
\]
Therefore $\la$ promotes to a representation of $\F\C X$ and we obtain the following commutative diagram
\[
\xymatrix{
\T(X)\simeq \ca_{\max}(\P\S X) \ar[rr]^{\la_\ast} \ar[dr]_{q_{\cov}^{\fock}} & & \T_\la(X) \\
& \T_{\cov}^{\fock}(X)\simeq \ca_{\max}(\F\C X) \ar[ur]_{\dot{\la}} &
}
\]
of canonical $*$-epimorphisms.
Since $\la_\ast$ and $q_{\cov}^{\fock}$ are equivariant we have that so is $\dot{\la}$.
On the other hand, since $\ker q_{\cov}^{\fock} \cap [\T(X)]_e = \ker \la_\ast \cap [\T(X)]_e$, we have that $\dot{\la}$ is injective on the $e$-fiber.
As the coaction by $G$ on $\la$ is normal, then so is the coaction by $G$ on $\dot{\la}$, and we conclude that
\[
\T_\la(X)=\ca(\dot{\la}) \simeq \ca_\la(\F\C X).
\]

The Fock representation satisfies the following property.
Let $p_i, q_i \in P$ and $\eps, \eps' \in \{0,1\}$ such that $p_1^{-\eps} q_1 \cdots p_n^{-1} q_n^{\eps'} = e$ and $q_n^{-\eps'}p_n \dots q_1^{-1}p_1^{\eps}P = \mt$.
Proposition \ref{P:fc mt} yields
\[
\left( \la_{p_1}(\xi_{p_1})^* \right)^{\eps} \la_{q_1}(\xi_{q_1}) \cdots \la_{p_n}(\xi_{p_n})^*\la_{q_n}(\xi_{q_n})^{\eps'} X_r = (0)
\foral 
r \in P,
\]
and therefore we have
\begin{equation} \label{eq:fc mt}
\bo{K}_{\mt, \la_\ast} = (0).
\end{equation}

We aim to give a characterisation for the equivariant Fock covariant injective representations of $X$.
We begin with the following proposition.

\begin{proposition}\label{P:B red}
Let $P$ be a unital subsemigroup of a discrete group $G$ and let $X$ be a product system over $P$.
Let $\hat{t}$ be a representation of $X$ such that $\T(X) = \ca(\hat{t})$, and let $t$ be an equivariant representation of a product system $X$. 
Then $t$ is Fock covariant if and only if 
\[
\ker \la_\ast\cap\bo{B}_{\F,\hat{t}_\ast}\subseteq\ker t_\ast\cap\bo{B}_{\F,\hat{t}_\ast}
\text{ for every finite $\cap$-closed $\F\subseteq \J$}.
\]
\end{proposition}

\begin{proof}
If $t$ is Fock covariant, then by definition $t_\ast$ factors through the canonical $*$-epimorphism $q_{\cov}^{\fock} \colon \T(X) \to \T_{\cov}^{\fock}(X)$, and thus $\ker q_{\cov}^{\fock} \subseteq \ker t_\ast$.
By the definition of $\J_{\cov}^{\fock}$ we also have $\ker \la_\ast \cap [\T(X)]_e = \ker q_{\cov}^{\fock} \cap [\T(X)]_e$, and therefore
\[
\ker \la_\ast \cap \bo{B}_{\F,\hat{t}_\ast} 
=
\ker q_{\cov}^{\fock} \cap \bo{B}_{\F,\hat{t}_\ast} 
\subseteq
\ker t_\ast \cap \bo{B}_{\F,\hat{t}_\ast},
\]
where we used that each $\bo{B}_{\F, \hat{t}_\ast}$ is a C*-subalgebra of $[\T(X)]_e$.

For the converse, recall that the fixed point algebra $[\T(X)]_e$ is the inductive limit of $\bo{B}_{\F,\hat{t}_\ast}$ for $\cap$-closed families $\F \subseteq \J$.
By the properties of the inductive limits we have
\begin{align*}
\ker q_{\cov}^{\fock} \cap [\T(X)]_e
& =
\ker \la_\ast \cap [\T(X)]_e
=
\ol{\bigcup_{\F}\left( \ker\la_\ast \cap \bo{B}_{\F,\hat{t}_\ast} \right)} 
\\
& \subseteq
\ol{\bigcup_{\F}\left(\ker t_\ast \cap \bo{B}_{\F,\hat{t}_\ast}\right)}
= 
\ker t_\ast\cap [\T(X)]_e,
\end{align*}
where we used the assumption for the inclusion.
By combining Proposition \ref{P:eq dgrm} with Proposition \ref{P:max psx} we get that $t$ induces a representation of $\F\C X$, i.e., $t$ is Fock covariant, and the proof is complete.
\end{proof}

We can now provide the characterisation of equivariant Fock covariant injective representations of $X$.

\begin{theorem}\label{T:Fock cov}
Let $P$ be a unital subsemigroup of a discrete group $G$ and let $X$ be a product system over $P$.
Let $\hat{t}$ be a representation of $X$ such that $\T(X) = \ca(\hat{t})$.
An equivariant injective representation $t$ of $X$ is Fock covariant if and only if $t$ satisfies the following conditions:
\begin{enumerate}
\item $\bo{K}_{\mt, t_\ast} = (0)$.
\item For any $\cap$-closed $\F=\{\bo{x}_1,\dots,\bo{x}_n\} \subseteq \J$ such that $\bigcup_{i=1}^n \bo{x}_i \neq \mt$, and any $b_{\bo{x}_i}\in \bo{K}_{\bo{x}_i,\hat{t}_\ast}$, with $i=1,\dots,n$, the following property holds:
\begin{center}
if $\sum\limits_{i: r\in \bo{x}_i}t_\ast(b_{\bo{x}_i})t_r(X_r)=(0)$ for all $r\in \bigcup_{i=1}^n \bo{x}_i$, then $\sum\limits_{i=1} ^n t_\ast(b_{\bo{x}_i})=0$.    
\end{center}
\end{enumerate}
\end{theorem}

\begin{proof} 
First we note that the Fock representation $\la$ satisfies conditions (i) and (ii).
Condition (i) for $\la$ is shown in (\ref{eq:fc mt}).
For condition (ii), let $\F=\{ \bo{x}_1,\dots,\bo{x}_n\}$ be a $\cap$-closed finite subset of $\J$ such that $\bigcup_{i=1}^n \bo{x}_i \neq \mt$, and let $b_{\bo{x}_i}$ be in $ \bo{K}_{\bo{x}_i,\hat{t}_\ast}$ such that 
\[
\sum_{i: r\in \bo{x}_i}\la_\ast(b_{\bo{x}_i})\la_r(\xi_r)=0
\foral
r\in \bigcup_{i=1}^n \bo{x}_i.
\]
Recall that we have $\la_\ast(b_{\Bx_i}) \xi_r = 0$ whenever $r \notin \Bx_i$ by Proposition \ref{P:fc mt 2}.
Therefore, if $r\notin\bigcup_{i=1}^n \bo{x}_i$, then $\sum_{i=1}^n\la_\ast(b_{\bo{x}_i}) \xi_r=0$.
On the other hand, if $r\in \bigcup_{i=1}^n \bo{x}_i$, then we have
\begin{align*}
\la_r\left(\sum_{i=1}^n\la_\ast(b_{\bo{x}_i})\xi_r\right)
=
\sum_{i: r\in \bo{x}_i}\la_r(\la_\ast(b_{\bo{x}_i})\xi_r)
=
\sum_{i: r\in \bo{x}_i}\la_\ast(b_{\bo{x}_i})\la_r(\xi_r)
=
0,
\end{align*}
where in the first equality we used that we have $\la_\ast(b_{\Bx_i}) \xi_r = 0$ whenever $r \notin \Bx_i$ from Proposition \ref{P:fc mt 2}, and in the second equality we used Proposition \ref{P:taut}.
By injectivity of $\la_r$, we have $\sum_{i=1}^n\la_\ast(b_{\bo{x}_i})\xi_r=0$. 
This concludes that $\sum_{i=1}^n\la_\ast(b_{\bo{x}_i}) = 0$, as required.

Now suppose that $t$ is Fock covariant.
Recall that we have
\[
\ker \la_\ast \cap [\T(X)]_g = \ker q_{\cov}^{\fock} \cap [\T(X)]_g
\foral g \in G,
\]
by the definition of $\J_{\cov}^{\fock}$.
Since $t$ is Fock covariant, we have that $t_\ast$ factors through the quotient map $q_{\cov}^{\fock} \colon \T(X) \to \T_{\cov}^{\fock}(X)$, and therefore
\[
\ker \la_\ast \cap [\T(X)]_g 
= 
\ker q_{\cov}^{\fock} \cap [\T(X)]_g
\subseteq
\ker t_\ast \cap [\T(X)]_g
\foral g \in G.
\]

For condition (i), we have
\[
\bo{K}_{\mt, \hat{t}_\ast} \subseteq \ker\la_\ast\cap[\T(X)]_e
\subseteq
\ker t_\ast\cap[\T(X)]_e,
\]
where we used that $\la$ satisfies condition (i) in the first inclusion.
It thus follows that
\[
\bo{K}_{\mt, t_\ast} = t_\ast(\bo{K}_{\mt, \hat{t}_\ast}) = (0),
\]
as required.

For condition (ii), let $\F=\{ \bo{x}_1,\dots,\bo{x}_n\}$ be a $\cap$-closed finite subset of $\J$ such that $\bigcup_{i=1}^n \bo{x}_i \neq \mt$, and let $b_{\bo{x}_i} \in \bo{K}_{\bo{x}_i,\hat{t}_\ast}$ such that  
\[
\sum_{i: r\in \bo{x}_i}t_\ast(b_{\bo{x}_i})t_r(X_r)=(0)
\foral 
r\in \bigcup_{i=1}^n \bo{x}_i.
\]
Fix $r \in P$.
If $r\not\in\bigcup_{i=1}^n \bo{x}_i$, then Proposition \ref{P:fc mt 2} yields
\[
\sum_{i=1}^n \la_\ast(b_{\bo{x}_i})X_r=(0).
\]
On the other hand, if $r\in\bigcup_{i=1}^n \bo{x}_i$ then for every $\xi_r\in X_r$ we have
\begin{align*}
t_r \left( \sum_{i=1}^n \la_\ast(b_{\bo{x}_i})\xi_r \right)
=
\sum_{i: r\in \bo{x}_i}t_r(\la_\ast(b_{\bo{x}_i})\xi_r)
= 
\sum_{i: r\in \bo{x}_i}t_\ast(b_{\bo{x}_i})t_r(\xi_r)
=
0,
\end{align*}
where in the first equality we used that we have $\la_\ast(b_{\Bx_i}) \xi_r = 0$ whenever $r \notin \Bx_i$ from Proposition \ref{P:fc mt 2}, and in the second equality we used Proposition \ref{P:taut}.
Injectivity of $t$ implies that $\sum_{i=1}^n \la_\ast(b_{\bo{x}_i})\xi_r=0$. We conclude that $\sum_{i=1}^n \la_\ast(b_{\bo{x}_i})=0$ and thus
\[
\sum_{i=1}^n b_{\bo{x}_i}\in\ker \la_\ast \cap [\T(X)]_e \subseteq \ker t_\ast \cap [\T(X)]_e.
\]
Therefore $\sum_{i=1}^n t_\ast(b_{\Bx_i}) = 0$, as required.
This completes the proof of the one direction.

For the converse, suppose that $t$ satisfies conditions (i) and (ii).
By Proposition \ref{P:B red}, it suffices to prove the inclusion $\ker \la_\ast \cap \bo{B}_{\F,\hat{t}_\ast}\subseteq\ker t_\ast \cap \bo{B}_{\F,\hat{t}_\ast}$ for every finite $\cap$-closed $\F\subseteq \J$.
If $\F = \{\mt\}$, then condition (i) for $\la$ and $t$ yields
\[
\ker \la_\ast \cap \bo{B}_{\{\mt\},\hat{t}_\ast}
=
\ker \la_\ast \cap \bo{K}_{\mt,\hat{t}_*} 
=
\bo{K}_{\mt,\hat{t}_*} 
=
\ker t_\ast \cap \bo{K}_{\mt,\hat{t}_*} 
=
\ker t_\ast\cap \bo{B}_{\{\mt\},\hat{t}_\ast},
\]
where we used the fact that the space $\bo{K}_{\mt,\hat{t}_*} = \bo{B}_{\{\mt\},\hat{t}_\ast}$ is a C*-subalgebra of $[\T(X)]_e$.
Next suppose that $\F=\{\bo{x}_1,\dots,\bo{x}_n\}$ such that $\bigcup_{i=1}^n \Bx_i \neq \mt$, and let $b_{\bo{x}_i} \in \bo{K}_{\bo{x}_i,\hat{t}_\ast}$ for $i=1,\dots, n$ such that
\[
\sum_{i=1}^n b_{\bo{x}_i} \in \ker \la_\ast \cap \bo{B}_{\F,\hat{t}_\ast}.
\]
For $r\in \bigcup_{i=1}^n \bo{x}_i$ we have
\[
\sum_{i: r\in \bo{x}_i} t_\ast(b_{\bo{x}_i})t_r(\xi_r)
=
t_r \left(\sum_{i: r\in \bo{x}_i} \la_\ast(b_{\bo{x}_i})\xi_r \right)
=
t_r\left(\sum_{i=1}^n\la_\ast(b_{\bo{x}_i})\xi_r\right)
=
0,
\] 
where in the first equality we used Proposition \ref{P:taut}, and on the second equality we used that we have $\la_\ast(b_{\bo{x}_i})\xi_r=0$ if $r\notin \bo{x}_i$ from Proposition \ref{P:fc mt 2}.
Condition (iii) now yields $\sum_{i=1}^n t_\ast(b_{\bo{x}_i})=0$, and thus
\[
\sum_{i=1}^n b_{\bo{x}_i} \in \ker t_\ast \cap \bo{B}_{\F,\hat{t}_\ast},
\]
as required.
\end{proof}

\begin{remark}
We note that, if $t$ is an equivariant Fock covariant injective representation of $X$, then the induced $*$-representation $\T_{\cov}^{\fock}(X) \to \ca(t)$ is injective on every $\bo{K}_{\bullet}$-core of $\T_{\cov}^{\fock}(X)$.
Indeed, 
let $b_{\Bx} \in \bo{K}_{\Bx, \hat{t}}$ for some $\mt \neq \Bx \in \J$ such that $\la_\ast(b_{\Bx})\not=0$, where $\ca(\hat{t}) = \T(X)$.
Then there exists a $\xi_r \in X_r$ such that $\la_\ast(b_{\Bx}) \xi_r \neq 0$.
By Proposition \ref{P:taut} we deduce that
\[
t_\ast(b_{\Bx}) t_r(\xi_r) = t_r(\la_\ast(b_{\Bx}) \xi_r) \neq 0,
\]
as $t$ is injective, and therefore $t_\ast(b_{\Bx}) \neq 0$.
\end{remark}

\subsection{Compactly aligned product systems over right LCM semigroups}\label{Ss:fock covariant}

A unital semigroup $P$ is said to be a \emph{right LCM semigroup} if it is left cancellative and satisfies Clifford's condition, i.e., for every $p,q\in P$ with $pP\cap qP \neq\mt$, there exists a $w\in P$ such that $pP\cap qP=wP$.
The element $w$ is referred to as a \emph{right least common multiple} or \emph{right LCM} of $p$ and $q$.
If an element  $w\in P$ is a right least common multiple of $p, q \in P$, then so is $w r$  for every $r \in P^*:= P\cap P^{-1}$.
It follows that $P$ is a right LCM semigroup if and only if $\J = \{pP \mid p \in P\} \bigcup \{\mt\}$.

In this section we will only consider group embeddable right LCM semigroups $P$ inside $G$.
If the right LCM is unique for every pair of elements $g_1, g_2 \in G$, then the pair $(G,P)$ is called \emph{quasi-lattice ordered}, a setup that was introduced by Nica \cite{Nic92}.
Although Nica's definition is implemented across the group, Exel \cite{Exe17} noted that it can be relaxed to be relevant only for the semigroup, in which case the pair $(G,P)$ is called \emph{weakly quasi-lattice ordered}.
Further examples of right LCM semigroups include the Artin monoids \cite{BS72}, the Baumslag-Solitar monoids $B(m,n)^+$ \cite{HNSY19, Li20, Spi12}, and the semigroup $R \rtimes R^\times$ of affine transformations of an integral domain $R$ that satisfies the GCD condition \cite{Li13, Nor14}.

Fowler's original work related to product systems over quasi-lattices \cite{Fow02}.
Product systems over right LCM semigroups were introduced and studied by Kwa\'{s}niewski and Larsen \cite{KL19a, KL19b}, extending the construction of Fowler, and they have been investigated further in \cite{DKKLL20, KKLL21b}.
The interest lies in that they retain several of the structural properties from the single C*-correspondence case.
With the addition of one further property, the generated C*-algebras admit a Wick ordering.
A product system $X$ over $P$ in the sense of Fowler gives rise to $*$-homomorphisms  
\[
i_p^{pq}\colon \L(X_p)\to \L (X_{pq}) 
\text{ where } 
i_p^{pq}(S) := u_{p, q}( S \otimes \id_{X_q} ) u_{p, q}^* \foral S \in \L(X_p).
\]
We say that $X$ is \emph{compactly aligned} if $i_{p}^{w}(k_p) i_q^w(k_q) \in \K(X_w)$ for all $k_p \in \K(X_p)$ and $k_q \in \K(X_q)$, whenever $pP \cap qP = wP$.
It was established in the discussion following \cite[Definition 2.9]{DKKLL20} that this definition is independent of the choice of $w$.
For a representation $t$ of a product system $X$ we use the notation
\[
t^{(p)} \colon \K(X_p) \to \ca(t); \theta^{X_p}_{\xi_p,\eta_p} \mapsto t_p(\xi_p) t_p(\eta_p)^*,
\]
for the induced $*$-representation of the compact operators $\K(X_p)$. 

\begin{proposition} \label{P:fock-nica}
Let $P$ be a unital right LCM semigroup that is a subsemigroup of a discrete group $G$ and let $X$ be a product system over $P$ in the sense of Fowler.
Then $X$ is compactly aligned if and only if for every $p, q \in P$ with $pP \cap qP = wP$ we have
\[
\la^{(p)}(\K(X_p)) \la^{(q)}(\K(X_q))
\subseteq 
\la^{(w)}(\K(X_w)).
\]
\end{proposition}

\begin{proof}
First suppose that $X$ is compactly aligned, and let $p, q \in P$ with $pP \cap qP = wP$.
Let $k_p \in \K(X_p)$ and $k_q \in \K(X_q)$. 
By compact alignment we have $i_p^w(k_p)i_q^w(k_q)\in \K(X_w)$. 
Hence for $r \in wP$ we get 
\begin{align*}
\la^{(p)}(k_p)\la^{(q)}(k_q)\xi_r
& =
\la^{(p)}(k_p)(i_q^r(k_q)\xi_r)
=
i_p^r(k_p)(i_q^r(k_q)\xi_r) \\
& =
i_w^r(i_p^w(k_p))i_w^r(i_q^w(k_q))\xi_r
=
i_w^r(i_p^w(k_p)i_q^w(k_q))\xi_r
=
\la^{(w)}(i_p^w(k_p)i_q^w(k_q))\xi_r.
\end{align*}
On the other hand for $r\not\in wP$, we have $\la^{(w)}(i_p^w(k_p)i_q^w(k_q))\xi_r = 0$.
If $r \notin qP$ then $\la^{(q)}(k_q)\xi_r = 0$.
If $r \in qP$ then $\la^{(q)}(k_q)\xi_r \in X_r$, but then it has to be that $r \notin pP$, and so $\la^{(p)}(k_p) (\la^{(q)}(k_q)\xi_r) = 0$.
In all cases, for $r \notin wP$ we have
\[
\la^{(p)}(k_p)\la^{(q)}(k_q)\xi_r
=
0
=
\la^{(w)}(i_p^w(k_p)i_q^w(k_q))\xi_r.
\]
Therefore, we conclude that
\[
\la^{(p)}(k_p)\la^{(q)}(k_q)=\la^{(w)}(i_p^w(k_p)i_q^w(k_q))\in \la^{(w)}(\K(X_w)).
\]

For the converse, let $p, q \in P$ such that $pP\cap qP=wP$ for some $w\in P$, and let $k_p \in \K(X_p)$ and $k_q \in \K(X_q)$.
Since
\[
\la^{(p)}(k_p) \la^{(q)}(k_q) \in \la^{(w)}(\K(X_w)),
\]
 there exists a $k_w \in\K(X_w)$ such that $\la^{(p)}(k_p)\la^{(q)}(k_q) = \la^{(w)}(k_w)$.
For every $\xi_w \in X_w$ we get
\[
i_p^w(k_p)i_q^w(k_q) \xi_w
=
\la^{(p)}(k_p)\la^{(q)}(k_q) \xi_w
= 
\la^{(w)}(k_w) \xi_w
=
k_w \xi_w.
\]
Therefore we have $i_p^w(k_p)i_q^w(k_q)=k_w\in \K(X_w)$, and the proof is complete.
\end{proof}

Let $P$ be a unital right LCM semigroup that is a subsemigroup of a discrete group $G$ and let $X$ be a compactly aligned product system over $P$ in the sense of Fowler. Let $t$ be a representation of $X$.
We say that $t$ is \emph{Nica covariant} if and only if
\[
t^{(p)}(k_p) t^{(q)}(k_q)
=
\begin{cases}
t^{(w)} \left( i_{p}^w(k_p) i_q^w(k_q) \right) & \text{if } pP \cap qP = wP, \\
0 & \text{if } pP\cap qP=\mt,
\end{cases}
\]
for all $k_p\in \K(X_p)$ and $k_q\in \K(X_q)$.
Because of linearity and continuity of $t^{(p)}, t^{(q)}$ and $t^{(w)}$, we have that $t$ is Nica covariant if and only if 
\[
t^{(p)} \left( \theta^{X_p}_{\xi_p,\eta_p} \right) t^{(q)} \left( \theta^{X_q}_{\xi_q,\eta_q} \right)
=
\begin{cases}
t^{(w)}\left( i_p^w\left(\theta^{X_p}_{\xi_p,\eta_p}\right) i_q^w\left(\theta^{X_q}_{\xi_q,\eta_q}\right) \right) & \text{if } pP\cap qP=wP, \\
0 & \text{if } pP\cap qP=\mt,
\end{cases}
\]
for all $\xi_p,\eta_p\in X_p$ and $\xi_q,\eta_q\in X_q$.
We write $\N\T(X)$ for the universal C*-algebra with respect to the Nica covariant representations of $X$.

In \cite[Proposition 2.4]{DKKLL20} it is shown that, if  $w \in P$ and $r \in P^*$, then $i_w^{wr}(k_w) \in \K(X_{wr})$ and $t^{(wr)}(i_{w}^{wr}(k_w)) = t^{(w)}(k_w) \foral k_w \in \K(X_w)$, when $t$ is Nica covariant.
Moreover, in the discussion following \cite[Definition 2.9]{DKKLL20} it is shown that
\begin{equation} \label{eq:Nc core1}
t_p(X_p)^*t_q(X_q)=(0) \text{ for } p,q\in P \text{ such that } pP\cap qP=\mt,
\end{equation}
and that
\begin{equation} \label{eq:Nc core2}
t_p(X_p)^*t_q(X_q)\subseteq [t_{r}(X_{r})t_{s}(X_{s})^*] \text{ for } wP=pP\cap qP \text{ and } r=p^{-1}w, s=q^{-1}w.
\end{equation}
Consequently, Nica covariance does not depend on the choice of a right least common multiple and also
 $\ca(t)$ admits a Wick ordering in the sense that
\[
\ca(t) = \ol{\spn} \{t_p(X_p) t_q(X_q)^* \mid p, q \in P\}.
\]
Moreover we have $\bo{K}_{pP, t_\ast} = t^{(p)}(\K(X_p))$.
Therefore, if $\F = \{p_1 P, \dots, p_n P\}$ is a finite $\cap$-closed subset of $\J$, then we have
\[
\bo{B}_{\F, t_\ast} = \sum_{i=1}^n t^{(p_i)}(\K(X_{p_i})),
\]
irrespectively of the choice of the elements $p_1, \dots, p_n$.
It is implicit in \cite{DKKLL20}, and proven in a more general context in \cite[Theorem 6.4]{KKLL21b}, that the cores $\bo{K}_{p_iP, \la_\ast}$ for distinct $p_i P$'s are linearly independent in $\bo{B}_{\F, \la_\ast}$.

It is shown in \cite[Proposition 3.5]{CLSV11} that $\N\T(X)$ admits a coaction since Nica covariance is a graded relation.
Let $\N X$ be the induced Fell bundle in $\N\T(X)$. 
By \cite[Proposition 4.3]{DKKLL20} we have
\[
\N\T(X) \simeq \ca_{\max}(\N X)
\qand
\T_{\la}(X) \simeq \ca_\la(\N X),
\]
by canonical $*$-isomorphisms.
In particular the second $*$-isomorphism induces a Fell bundle isomorphism between $\N X$ and $\F\C X$.
This shows that a representation of $X$ is Fock covariant if and only if it is Nica covariant.
Below we will give an alternative proof that squares with our characterisation of Fock covariance.
We will need the following proposition.

\begin{proposition}\label{P:Nc mt}
Let $P$ be a unital right LCM semigroup that is a subsemigroup of a discrete group $G$ and let $X$ be a compactly aligned product system over $P$ in the sense of Fowler. 
If $t$ is a Nica covariant representation of $X$, then
\[ 
\left(t_{p_1}(X_{p_1})^*\right)^\eps t_{q_1}(X_{q_1})\cdots t_{p_n}(X_{p_n})^*t_{q_n}(X_{q_n})^{\eps'}=(0)
\]
for $p_1, q_1, \dots, p_n, q_n \in  P$ and $\eps, \eps' \in \{0,1\}$ such that $q_n^{-\eps'} p_n \dots q_1^{-1} p_1^{\eps} P=\mt$. In particular, we have
$\bo{K}_{\mt,t_\ast}=(0)$.
\end{proposition}

\begin{proof}
We first consider the case $\eps=\eps'=1$.
In this case we have $q_n^{-1}p_n\dots q_1^{-1}p_1P=\mt$ and we wish to show that
\[
t_{p_1}(X_{p_1})^*t_{q_1}(X_{q_1})\cdots t_{p_n}(X_{p_n})^*t_{q_n}(X_{q_n})=(0).
\]
We proceed in steps.
If $p_n P\cap q_n P=\mt$, then Nica covariance of $t$ implies that
\[
t_{p_n}(X_{p_n})^*t_{q_n}(X_{q_n})=(0)
\]
by equation (\ref{eq:Nc core1}), which gives the desired result. 
If $p_n P\cap q_n P\not=\mt$, then choose $w_1 \in P$ with $p_nP\cap q_nP  = w_1P$, and (\ref{eq:Nc core2}) implies that
\[
t_{p_n}(X_{p_n})^*t_{q_n}(X_{q_n})
\subseteq 
[t_{r_1}(X_{r_1})t_{s_1}(X_{s_1})^*] 
\qfor
r_1=p_n^{-1}w_1, 
s_1=q_n^{-1}w_1.
\]
Hence we obtain the inclusion
\begin{align*}
&t_{p_1}(X_{p_1})^*t_{q_1}(X_{q_1})\cdots t_{p_n}(X_{p_n})^*t_{q_n}(X_{q_n}) \subseteq \\
& \hspace{3cm} \subseteq
[t_{p_1}(X_{p_1})^*t_{q_1}(X_{q_1})\cdots 
t_{p_{n-1}}(X_{p_{n-1}})^*t_{q_{n-1}}(X_{q_{n-1}})
t_{r_1}(X_{r_1})t_{s_1}(X_{s_1})^*]\\
& \hspace{3cm} \subseteq
[t_{p_1}(X_{p_1})^*t_{q_1}(X_{q_1})\cdots 
t_{p_{n-1}}(X_{p_{n-1}})^*
t_{q_{n-1}r_1}(X_{q_{n-1}r_1})t_{s_1}(X_{s_1})^*].
\end{align*}
If $p_{n-1} P\cap q_{n-1}r_1 P=\mt$, then Nica covariance gives the desired result as
\[
t_{p_{n-1}}(X_{p_{n-1}})^*
t_{q_{n-1}r_1}(X_{q_{n-1}r_1}) = (0).
\]
If $p_{n-1} P\cap q_{n-1}r_1 P\not=\mt$, then choose $w_2 \in P$ such that $p_{n-1}P\cap q_{n-1}r_1P = w_2P$, and (\ref{eq:Nc core2}) implies that
\[
t_{p_{n-1}}(X_{p_{n-1}})^*t_{q_{n-1}r_1}(X_{q_{n-1}r_1})\subseteq [t_{r_2}(X_{r_2})t_{s_2}(X_{s_2})^*]
\qfor
r_2=p_{n-1}^{-1}w_2,
s_2=r_1^{-1}q_{n-1}^{-1}w_2.
\]
Hence we obtain the inclusion
\begin{align*}
&t_{p_1}(X_{p_1})^*t_{q_1}(X_{q_1})\cdots t_{p_n}(X_{p_n})^*t_{q_n}(X_{q_n}) \subseteq \\
& \hspace{3cm} \subseteq
[t_{p_1}(X_{p_1})^*t_{q_1}(X_{q_1})\cdots 
t_{p_{n-1}}(X_{p_{n-1}})^*
t_{q_{n-1}r_1}(X_{q_{n-1}r_1})t_{s_1}(X_{s_1})^*]\\
& \hspace{3cm} \subseteq 
[t_{p_1}(X_{p_1})^*t_{q_1}(X_{q_1})\cdots 
t_{q_{n-2}}(X_{q_{n-2}})
t_{r_2}(X_{r_2})t_{s_2}(X_{s_2})^*t_{s_1}(X_{s_1})^*]\\
& \hspace{3cm} \subseteq 
[t_{p_1}(X_{p_1})^*t_{q_1}(X_{q_1})\cdots 
t_{p_{n-2}}(X_{p_{n-2}})^*
t_{q_{n-2}r_2}(X_{q_{n-2}r_2})t_{s_1s_2}(X_{s_1s_2})^*].
\end{align*}
Continuing inductively we either get a zero space at a step, or we have obtained elements $w_k, r_k, s_k \in P$, with $k=2,\dots,n-1$, such that
\begin{align*}
w_kP & = p_{n-(k-1)}P \cap q_{n-(k-1)}r_{k-1} P, \,
r_k = p_{n-(k-1)}^{-1}w_k, \,
s_k = r_{k-1}^{-1}q_{n-(k-1)}^{-1}w_k,
\end{align*}
and the inclusion
\[
t_{p_1}(X_{p_1})^*t_{q_1}(X_{q_1})\cdots t_{p_n}(X_{p_n})^*t_{q_n}(X_{q_n})
\subseteq
[t_{p_1}(X_{p_1})^*t_{q_1r_{n-1}}(X_{q_1r_{n-1}})t_{s_1\cdots s_{n-1}}(X_{s_1\cdots s_{n-1}})^*].
\]
In the latter case we claim that
\[
p_1P \cap q_1 r_{n-1}P = \mt,
\]
and so Nica covariance will imply the required identity.
In order to reach a contradiction suppose that $p_1P \cap q_1 r_{n-1}P \neq \mt$, so there exists a $q \in P$ such that $p := r_{n-1}^{-1} q_1^{-1} p_1 q \in P$.
By the construction of the $r_k$ and the $s_k$ for $k=1, 2, \dots, n-1$ we have
\begin{align*}
s_1 & = q_n^{-1}w_1 = q_n^{-1} p_n r_1 \\
s_2 & = r_1^{-1} q_{n-1}^{-1} w_2 = r_1^{-1} q_{n-1}^{-1} p_{n-1} r_2 \\
& \vdots \\
s_k & = r_{k-1}^{-1} q_{n-(k-1)}^{-1} w_k = r_{k-1}^{-1} q_{n-(k-1)}^{-1} p_{n-(k-1)} r_k \\
& \vdots \\
s_{n-1} & = r_{n-2}^{-1} q_{2}^{-1} w_{n-1} = r_{n-2}^{-1} q_2^{-1} p_2 r_{n-1}.
\end{align*}
Therefore for every $k=1, 2 \dots, n-1$ we have
\begin{align*}
s_1 s_2 \cdots s_k \cdots s_{n-1} p
& =
q_n^{-1} p_n q_{n-1}^{-1} p_{n-1} \cdots q_{n-(k-1)}^{-1} p_{n-(k-1)} \cdot \left(r_k s_{k+1} \cdots s_{n-1} p \right)\\
& \in 
q_n^{-1} \cdot p_n \cdots q_{n-(k-1)}^{-1} \cdot p_{n-(k-1)} \cdot P.
\end{align*}
Because of the choice of $p = r_{n-1}^{-1} q_1^{-1} p_1 q$ we also have
\begin{align*}
s_1  \cdots s_{n-2}s_{n-1} p
& =
q_n^{-1} p_n \cdots q_2^{-1} p_2 r_{n-1} \cdot \left(r_{n-1}^{-1} q_1^{-1} p_1 q\right) 
\in 
q_n^{-1} \cdot p_n \cdots q_1^{-1} \cdot p_1 \cdot P.
\end{align*}
Therefore we have $s_1  \cdots s_{n-2} s_{n-1} p \in q_n^{-1}p_n \dots q_1^{-1} p_1 P = \mt$ which is a contradiction, and the proof for the first case is complete.

Suppose now that $\eps=1$ and $\eps'=0$. We have $p_n q_{n-1}^{-1}\dots q_1^{-1}p_1P=\mt$ and we wish to show that
\[
t_{p_1}(X_{p_1})^*t_{q_1}(X_{q_1})\cdots t_{q_{n-1}}(X_{q_{n-1}})t_{p_n}(X_{p_n})^*=(0).
\]
Note that we must also have $q_{n-1}^{-1}p_{n-1}\dots q_1^{-1}p_1P=\mt$, and therefore the previous case implies that
\[
t_{p_1}(X_{p_1})^*t_{q_1}(X_{q_1})\cdots t_{p_{n-1}}(X_{p_{n-1}})^* t_{q_{n-1}}(X_{q_{n-1}})=(0),
\]
which completes the proof of this case.

Suppose now that $\eps=0$ and $\eps'=1$. We have $q_n^{-1}p_n\dots p_2q_1^{-1}P=\mt$ and we wish to show that
\[
t_{q_1}(X_{q_1})t_{p_2}(X_{p_2})^*\cdots t_{p_n}(X_{p_n})^*t_{q_n}(X_{q_n})=(0).
\]
Since $q_n^{-1}p_n\dots q_2^{-1}p_2P=q_n^{-1}p_n\dots p_2q_1^{-1}P=\mt$, the first case implies that
\[
t_{p_2}(X_{p_2})^*t_{q_2}(X_{q_2})\cdots t_{p_n}(X_{p_n})^*t_{q_n}(X_{q_n})=(0),
\]
which completes the proof of this case.

Finally, suppose that $\eps=\eps'=0$. We have $p_nq_{n-1}^{-1}\dots p_2q_1^{-1}P=\mt$ and we wish to show that
\[
t_{q_1}(X_{q_1})t_{p_2}(X_{p_2})^*\cdots t_{q_{n-1}}(X_{q_{n-1}}) t_{p_n}(X_{p_n})^*=(0).
\]
Since $p_nq_{n-1}^{-1}\dots q_2^{-1}p_2P=p_nq_{n-1}^{-1}\dots p_2q_1^{-1}P=\mt$, the second case implies that
\[
t_{p_2}(X_{p_2})^*t_{q_2}(X_{q_2})\cdots t_{q_{n-1}}(X_{q_{n-1}}) t_{p_n}(X_{p_n})^*=(0),
\]
and the proof is complete.
\end{proof}

We now pass to the connection between Fock covariance and Nica covariance through our characterisation.

\begin{proposition} \label{P:nc is fc}
Let $P$ be a unital right LCM semigroup that is a subsemigroup of a discrete group $G$ and let $X$ be a compactly aligned product system over $P$ in the sense of Fowler.
An equivariant injective representation of $X$ is Fock covariant if and only if it is Nica covariant.
\end{proposition}

\begin{proof}
Fix the representation $\hat{t}$  of $X$ such that $\T(X)=\ca(\hat{t})$.
First suppose that $t$ is an equivariant Fock covariant injective representation of $X$, so that it satisfies conditions (i) and (ii) of Theorem \ref{T:Fock cov}.
Fix $p, q \in P$.
We will show that
\[
t_p(\xi_p) t_p(\eta_p)^* t_q(\xi_q) t_q(\eta_q)^*
=
\begin{cases}
t^{(w)} \left( i_p^w\left(\theta^{X_p}_{\xi_p,\eta_p}\right) i_q^w\left(\theta^{X_q}_{\xi_q,\eta_q}\right) \right) & \text{if } pP\cap qP=wP, \\
0 & \text{if } pP\cap qP=\mt,
\end{cases}
\]
for all $\xi_p,\eta_p\in X_p$ and $\xi_q,\eta_q\in X_q$.
Towards this end, set
\[
\bo{x} := q q^{-1}p p^{-1} P = pP\cap qP.
\]
If $\bo{x}=\mt$, then condition (i) of Theorem \ref{T:Fock cov} implies that
\[
t_p(\xi_p)t_p(\eta_p)^*t_q(\xi_q)t_q(\eta_q)^*
\in
\bo{K}_{\mt, t_\ast} = (0),
\]
as required.
On the other hand, if $\bo{x}=wP$ for some $w\in P$ set
\[
b_{\Bx} := \hat{t}_p(\xi_p) \hat{t}_p(\eta_p)^* \hat{t}_q(\xi_q) \hat{t}_q(\eta_q)^*
\qand
k_w := i_p^w\left(\theta^{X_p}_{\xi_p,\eta_p}\right) i_q^w\left(\theta^{X_q}_{\xi_q,\eta_q}\right) \in \K(X_w).
\]
We have to show that $t_\ast(b_{\Bx}) = t^{(w)}\left( k_w \right)$.
We will use the equality $t^{(w)}(k_w) = t_\ast(\hat{t}^{(w)}( k_w))$.
Since $\Bx = qq^{-1}pp^{-1}P$ we have $b_{\Bx} \in \bo{K}_{\Bx, \hat{t}_\ast}$.
Moreover, since $\Bx = wP = ww^{-1}P$ we have
\[
\hat{t}^{(w)}(k_w) \in [\hat{t}_w(X_w)\hat{t}_w(X_w)^*]\subseteq \bo{K}_{\bo{x},\hat{t}_\ast},
\]
and therefore
\[
t_\ast(b_{\Bx})-t_\ast(\hat{t}^{(w)}( k_w)) \in \bo{K}_{\Bx, t_\ast}.
\]
Fix $r \in \bo{x}$.
Then we have
\begin{align*}
\la_\ast(b_{\Bx})\xi_r
& =
\la^{(p)}\left(\theta^{X_p}_{\xi_p,\eta_p}\right) \la^{(q)}\left(\theta^{X_q}_{\xi_q,\eta_q}\right)\xi_r
=
i_p^r\left(\theta^{X_p}_{\xi_p,\eta_p}\right) i_q^r\left(\theta^{X_q}_{\xi_q,\eta_q}\right)\xi_r
=
i_w^r\left( k_w \right) \xi_r,
\end{align*}
and thus by Proposition \ref{P:taut} we have
\[
t_\ast(b_{\Bx})t_r(\xi_r)
=
t_r ( \la_\ast(b_{\Bx})\xi_r)
=
t_r \left( i_w^r \left( k_w \right) \xi_r \right).
\]
On the other hand, by Proposition \ref{P:taut} we have
\begin{align*}
t^{(w)}(k_w)t_r(\xi_r)
=
t_\ast(\hat{t}^{(w)}(k_w)) t_r(\xi_r)
=
t_r(\la_\ast(\hat{t}^{(w)}(k_w))\xi_r)
=
t_r(\la^{(w)}(k_w)\xi_r)
=
t_r(i_w^r(k_w)\xi_r).
\end{align*}
Therefore we conclude that
\[
\left( t_\ast(b_{\Bx})-t_\ast(\hat{t}^{(w)}( k_w)) \right) t_r(\xi_r) 
= 
\left( t_\ast(b_{\Bx})-t^{(w)}(k_w) \right) t_r(\xi_r) 
=
0
\foral
r \in \Bx.
\]
Condition (ii) for $t$ then implies that $t_\ast(b_{\Bx}) = t_\ast(\hat{t}^{(w)}( k_w)) = t^{(w)}(k_w)$, as required.

For the converse, suppose that $t$ is an equivariant Nica covariant injective representation of $X$, and we will show that $t$ satisfies conditions (i) and (ii) of Theorem \ref{T:Fock cov}.
The fact that $t$ satisfies condition (i) is already verified in Proposition \ref{P:Nc mt}. 
We will show that $t$ satisfies condition (ii).
In order to fix notation let the commutative diagrams of canonical $*$-epimorphisms
\[
\xymatrix{
\T(X) \ar[rr]^{\la_\ast} \ar[dr]_{q_{\cov}^{\rm N}} & & \T_\la(X) \\
& \N\T(X) \ar[ur]_{\dot{\la}} &
}
\qand
\xymatrix{
\T(X) \ar[rr]^{t_\ast} \ar[dr]_{q_{\cov}^{\rm N}} & & \ca(t) \\
& \N\T(X) \ar[ur]_{\dot{t}} &
}
\]
induced by $\la$ and $t$, and let $\tilde{t}$ be the representation of $X$ such that $\N\T(X) = \ca(\tilde{t})$.
We will use the following note for an element $b \in \bo{K}_{pP, \hat{t}_\ast}$. 
Since
\[
q_{\cov}^{\rm{N}}(b) \in q_{\cov}^{\rm N}(\bo{K}_{pP, \hat{t}_\ast}) = \bo{K}_{pP, \tilde{t}_\ast} = \tilde{t}^{(p)}(\K(X_p)),
\]
there exists a $k_p \in \K(X_p)$ such that $q_{\cov}^{\rm{N}}(b)=\tilde{t}^{(p)}(k_p)$.
In particular we have
\[
t_\ast(b) 
= 
\dot{t} \circ q_{\cov}^{\rm{N}}(b)
=
\dot{t}(\tilde{t}^{(p)}(k_p))
=
t^{(p)}(k_p).
\]
Likewise we have
\[
\la_\ast(b) 
= 
\dot{\la} \circ q_{\cov}^{\rm{N}}(b)
=
\dot{\la}(\tilde{t}^{(p)}(k_p))
=
\la^{(p)}(k_p).
\]
 
For condition (ii), first let $\F=\{\bo{x}_1,\dots,\bo{x}_n\}$ be a $\cap$-closed subset of $\J$ where $\bo{x}_i=p_iP$ and $b_{\bo{x}_i}\in\bo{K}_{\bo{x}_i,\hat{t}_\ast}$ for $i=1,\dots,n$.
Note that without loss of generality we may suppose that $\Bx_i\not=\Bx_j$ for $i\not=j$ and hence the C*-subalgebras $\bo{K}_{p_iP,\la_\ast}$'s are linearly independent in $\bo{B}_{\F,\la_\ast}$, by \cite[Theorem 6.4]{KKLL21b}.
Suppose that
\[
\sum_{i:r\in \bo{x}_i}t_\ast(b_{\bo{x}_i})t_r(X_r)=(0)
\foral
r\in\bigcup_{i=1}^n\bo{x}_i.
\]
We will show that $q_{\cov}^{\rm{N}}(b_{\bo{x}_i}) = 0$ for all $i=1, \dots, n$, and thus
\[
\sum_{i=1}^n t_\ast(b_{\bo{x}_i}) = \sum_{i=1}^n \dot{t}\left( q_{\cov}^{\rm{N}}(b_{\bo{x}_i}) \right) = 0.
\]
From the comments above, for each $i=1,\dots, n$ we may pick a $k_{p_i} \in \K(X_{p_i})$ such that $q_{\cov}^{\rm{N}} (b_{\bo{x}_i}) = \tilde{t}^{(p_i)}(k_{p_i})$.
Then $\la_\ast(b_{\Bx_i}) = \la^{(p_i)}(k_{p_i})$ for all $i=1, \dots, n$, and thus for each $r\in\bigcup_{i=1}^n\bo{x}_i$ we have
\begin{align*}
t_r\left(\sum_{i=1}^n\la^{(p_i)}(k_{p_i})\xi_r\right)
=
\sum_{i:r\in\bo{x}_i}t_r(\la^{(p_i)}(k_{p_i})\xi_r)
=
\sum_{i:r\in\bo{x}_i}t_r(\la_\ast(b_{\bo{x}_i})\xi_r)
=
\sum_{i:r\in\bo{x}_i}t_\ast(b_{\bo{x}_i})t_r(\xi_r)
= 
0,
\end{align*}
by using Proposition \ref{P:taut}.
Injectivity of $t_r$ now yields $\sum_{i=1}^n\la^{(p_i)}(k_{p_i})\xi_r=0$ for all $r\in\bigcup_{i=1}^n\bo{x}_i$.
On the other hand  we have $\la^{(p_i)}(k_{p_i})\xi_r=0$ for all $r\not\in\bigcup_{i=1}^n\bo{x}_i$ and $i=1,\dots, n$, and therefore we obtain
\[
\sum_{i=1}^n\la^{(p_i)}(k_{p_i})=0.
\]
As the $\bo{K}_{\bullet}$-cores are linearly independent by \cite[Theorem 6.4]{KKLL21b}, we deduce that $\la^{(p_i)}(k_{p_i})=0$ for all $i=1, \dots, n$,
and thus $k_{p_i}=0$ for all $i=1, \dots, n$, from the injectivity of $\la^{(p_i)}$. 
In particular we have
\[
q_{\cov}^{\rm{N}}(b_{\bo{x}_i}) = \tilde{t}^{(p_i)}(k_{p_i}) = 0
\foral
i=1, \dots, n,
\]
as required.
To finish the proof, next let $\F =\{\Bx_1, \dots, \Bx_n\} \cup \{\mt\}$ such that $\Bx_i\not= \Bx_j$ for $i\not = j$, and $\Bx_i = p_i P$ for some $p_i \in P$.
Let $b_{\mt}, b_{\Bx_1}, \dots, b_{\Bx_n}$ where $b_{\mt}\in \bo{K}_{\mt,\hat{t}_\ast}$ and $b_{\Bx_i}\in \bo{K}_{\Bx_i,\hat{t}_\ast}$ for $i=1, \dots, n$, such that 
\[
\sum_{i : r \in \Bx_i} t_\ast(b_{\Bx_i}) t_r(X_r) = (0)
\foral
r \in \mt \cup \left(\bigcup_{i=1}^n \Bx_i\right) = \bigcup_{i=1}^n \Bx_i.
\]
Then the previous arguments show that $q_{\cov}^{\rm{N}}(b_{\Bx_i}) = 0$ for all $i=1, \dots, n$.
Moreover by Proposition \ref{P:Nc mt} for $\tilde{t}$ we also have $q_{\cov}^{\rm{N}}(b_\mt) = 0$.
Hence we get
\[
t_\ast(b_\mt) + \sum_{i=1}^n t_\ast(b_{\bo{x}_i}) 
= 
\dot{t}(q_{\cov}^{\rm{N}}(b_\mt)) + \sum_{i=1}^n \dot{t}( q_{\cov}^{\rm{N}}(b_{\bo{x}_i}) ) 
= 
0,
\]
as required, and the proof is complete.
\end{proof}

As a corollary we have an alternative proof of \cite[Proposition 4.3]{DKKLL20}.

\begin{corollary}
Let $P$ be a unital right LCM semigroup that is a subsemigroup of a discrete group $G$ and let $X$ be a compactly aligned product system over $P$ in the sense of Fowler. Then $\T_{\cov}^{\fock}(X)$ and $\N\T(X)$ are canonically isomorphic. In particular, a representation of $X$ is Fock covariant if and only if it is Nica covariant.
\end{corollary}

\begin{proof}
This is immediate by Proposition \ref{P:nc is fc} since the universal Fock covariant and the universal Nica covariant representations are both equivariant and injective.
\end{proof}

\subsection{The product system of a semigroup}\label{Ss:fock covariant}

One class of product systems of particular interest arises from the semigroup representations of $P$ itself.
There is a well-known left regular model in this case considered by Li \cite{Li13}, giving rise to a Fock covariant Fell bundle.
The corresponding graded relations were identified by Laca and Sehnem in \cite[Section 3]{LS21}, although not explicitly stated like this, see also \cite[Remark 3.15]{KKLL21a}.
In this subsection we will show how the Fock covariance description of \cite[Section 3]{LS21} coincides with our general description from Theorem \ref{T:Fock cov} in two ways.
One is achieved categorically though equation (\ref{eq:fock is u}), and the other is achieved from first principles in Proposition \ref{P:fock is u}, thus providing an independent validation of our results.

We consider $P$ to be a unital subsemigroup of a discrete group $G$.
Following the notation of \cite{LS21}, we write $\al = (p_1, q_1, \dots, p_n, q_n)$ for the words of even length where $p_1, q_1, \dots, p_n, q_n\in P$.
A word $\al = (p_1, q_1, \dots, p_n, q_n)$ is called \emph{neutral} if $p_1^{-1} q_1 \cdots p_n^{-1} q_n = e$.
For a word $\al = (p_1, q_1, \dots, p_n, q_n)$ we write
\[
K(\al) := q_n^{-1} p_n \dots q_1^{-1} p_1 P
\]
for the induced constructible ideal in $\J$.
For a map $w \colon P \to \B(H)$ we write
\[
\dot{w}_\al := w_{p_1}^* w_{q_1} \cdots w_{p_n}^* w_{q_n} \in \B(H).
\]
It follows that
\[
\dot{w}_\al \dot{w}_\be = \dot{w}_{\al \be},
\text{ for all words $\al, \be$,}
\]
where $\al \be$ denotes the concatenation of the words $\al$ and $\be$.

A map $w\colon P \to \B(H)$ is called a \emph{representation of $P$} if it is a semigroup homomorphism.
A representation $w$ of $P$ will be called \emph{equivariant} if there exists a $*$-homomorphism $\de$ of $\ca(w)$ such that
\[
\de \colon \ca(w) \to \ca(w) \otimes \ca_{\max}(G) ; w_p \mapsto w_p \otimes u_{p}.
\]
It follows that $\de$ is injective with a left inverse given by the map $\id \otimes \chi$.
Moreover, it satisfies the coaction identity and hence $\ca(w)$ admits a coaction by $G$.
For simplicity, we will say that $w$ admits a coaction by $G$ if such a $\de$ exists.

\begin{remark}
Let $w\colon P \to \B(H)$ be an isometric representation of $P$ (i.e., every $w_p$ is an isometry) and let $\al = (p_1, q_1, \dots, p_n, q_n)$ where $K(\al)\not = \mt$. 
Since every $w_p$ is an isometry we have
\[
w_p^* w_{pq} = w_p^* w_p w_q = w_q,
\foral
p,q \in P.
\]
Therefore a variant of the proof of Proposition \ref{P:mult} yields
\[
\dot{w}_\al w_r = w_{p_1^{-1} q_1 \cdots p_n^{-1}q_n r} 
\foral r \in K(\al).
\]
In particular, if $\al$ is neutral we obtain
\[
\dot{w}_\al w_r = w_r 
\foral r \in K(\al).
\]
\end{remark}
In \cite[Proposition 3.2]{LS21} it is shown that, if a map $w \colon P \to \B(H)$ satisfies
\begin{itemize}
\item [\rm{(T1)}] $w_{e}=1$,
\item [\rm{(T2)}] $\dot{w}_{\al}=0$ if $K(\al)=\mt$ for a neutral word $\al$, and
 \item [\rm{(T3)}] $\dot{w}_{\al}=\dot{w}_{\be}$ if $K(\al)=K(\be)$ for neutral words $\al$ and $\be$,
 \end{itemize}
then $w$ is an isometric representation of $P$, and the operators $\{\dot{w}_\al\}_{\al: \text{neutral}}$ are commuting projections. 
Let $\T(P)$ be the universal C*-algebra with respect to the isometric representations of $P$.
By the universal property, the C*-algebra $\T(P)$ admits a coaction by $G$ and hence a topological C*-grading. 

Let $\T_u(P)$ be the universal C*-algebra with respect to the maps $w\colon P\to \B(H)$ that satisfy conditions (T1)--(T3) and the additional condition:
\begin{itemize}
\item [\rm{(T4)}] $\prod_{\be\in F}(\dot{w}_{\al}-\dot{w}_{\be})=0$ whenever $\al$ is a neutral word, $F$ is a finite set of neutral words and $K(\al)=\bigcup_{\be\in F}K(\be)$.
\end{itemize}
Since conditions (T1)--(T4) are graded with respect to the grading of $\T(P)$, the C*-algebra $\T_u(P)$ is a quotient of $\T(P)$ by an induced ideal, and therefore $\T_u(P)$ admits a coaction by $G$. By using similar arguments as in Proposition \ref{P:max psx} it can be proved that we have
\[
\ca_{\max}\left(\Big\{[\T(P)]_g\Big\}_{g\in G}\right)\simeq \T(P),
\]
and hence combining with Proposition \ref{P:qnt bnd} we obtain
\[
\ca_{\max}\left(\Big\{[\T_u(P)]_g\Big\}_{g\in G}\right)\simeq \T_u(P).
\]
In \cite[Corollary 3.19, Corollary 3.20]{LS21} it is shown that there is a $*$-epimorphism $\T_u(P) \to \ca_\la(P)$ that is equivariant and injective on the fixed point algebra.
Hence $\T_u(P)$ coincides with the universal C*-algebra of the Fell bundle induced in $\ca_\la(P)$; see also \cite[Remark 3.15]{KKLL21a}.

Fix $\{v_p\}_{p \in P}$ such that $\T(P) = \ca(v)$.
We can then define the concrete product system $X_P$ in $\T(P)$ by
\[
X_{P, p} := \bC v_p \foral p \in P.
\]
We say that a representation $t$ of $X_P$ is \emph{unital} if it satisfies $t_e(v_e)=1$. Note that a unital representation is automatically injective. 
If $t\colon X_P\to \B(H)$ is a non-zero representation, then $t_e(v_e)$ is a projection that commutes with $\ca(t)$ and hence $K:=t_e(v_e) H$ is reducing for $\ca(t)$. Therefore by compressing on $K$ we can consider $t$ to be a unital representation of $X_P$.
Note that the Fock space of $X_P$ is unitarily equivalent to $\ell^2(P)$ by the unitary
\begin{equation} \label{eq:wp un}
W_P \colon \F X_P \to \ell^2(P); v_p \mapsto \de_p.
\end{equation}
It follows that
\begin{equation} \label{eq:tl}
\T_\la(X_P)  \simeq \ca_\la(P)
\end{equation}
by the canonical $*$-isomorphism induced by $\ad_{W_P}$, since $W_P \la_p(v_p) W_P^* = V_p$ for all $p \in P$.
Since this $*$-isomorphism is canonical we have that the Fell bundles induced in $\T_{\la}(X_P)$ and $\ca_\la(P)$ are isomorphic, and therefore we also have
\begin{equation} \label{eq:fock is u}
\T_{\cov}^{\fock}(X_P) \simeq \T_u(P).
\end{equation}
More generally, the association
\[
t = \{t_p\}_{p \in P} \mapsto w_t = \{t_p(v_p)\}_{p \in P}
\]
induces a bijection between the unital (equivariant) representations of $X_P$ and the unital (resp. equivariant) isometric representations of $P$.
Hence we have 
\begin{equation}
\T(X_P) \simeq \T(P)
\end{equation}
by a canonical $*$-isomorphism.
Note that Laca and Sehnem \cite[Definition 3.6]{LS21} coin $\T_u(P)$ as the Toeplitz algebra of $P$; however we will not use this terminology, as $\T_u(P)$ is not $\T(X_P)$.
Below we give an alternative proof of (\ref{eq:fock is u}) that squares with our characterisation of Fock covariance.

\begin{proposition} \label{P:fock is u}
Let $P$ be a unital subsemigroup of a discrete group $G$ and let $X_P$ be the induced product system in $\T(P)$.
Then the association
\[
t = \{t_p\}_{p \in P} \mapsto w_t = \{t_p(v_p)\}_{p \in P}
\]
defines a bijection between the unital equivariant Fock covariant representations of $X_P$ and the equivariant representations of $P$ that satisfy conditions (T1)--(T4).
\end{proposition}

\begin{proof}
To fix notation, let $\T(P) = \ca(v)$, $\T(X_P) = \ca(\hat{t})$ and $\T_u(P)=\ca(\tilde{u})$.
It is readily verified that the association $t \mapsto w_t$ is a bijection between the unital equivariant representations of $X_P$ and the unital equivariant isometric representations of $P$. 
When it is clear from the context we will simply write $w$ instead of $w_t$.
With this notation we have
\[
t_\ast \left( \hat{t}_{p_1}(v_{p_1})^*\hat{t}_{q_1}(v_{q_1})\cdots\hat{t}_{p_n}(v_{p_n})^* \hat{t}_{q_n}(v_{q_n}) \right) = \dot{w}_\al
\]
for every  word $\al = (p_1, q_1, \dots, p_n, q_n)$ where $p_1, q_1, \dots, p_n, q_n\in P$.

First suppose that $t$ is a unital equivariant Fock covariant representation of $X_P$, and let $w$ be the associated unital equivariant representation of $P$.
Then $t$ satisfies conditions (i) and (ii) of Theorem \ref{T:Fock cov}.
For condition (T1), we have that $t_e(v_{e})$ is the unit, and thus $w_{e} = t_e(v_{e}) = 1$.

For condition (T2), let $\al=(p_1,q_1,\dots,p_n,q_n)$ be a neutral word such that $K(\al)=\mt$.
We then have
\[
\dot{w}_\al
=
t_\ast \left( \hat{t}_{p_1}(v_{p_1})^*\hat{t}_{q_1}(v_{q_1})\cdots\hat{t}_{p_n}(v_{p_n})^* \hat{t}_{q_n}(v_{q_n}) \right) 
\in
\bo{K}_{\mt, t_\ast}
=(0),
\]
as required.

For condition (T3), let $\al=(p_1,q_1,\dots,p_n,q_n)$ and $\be=(r_1,s_1,\dots,r_m,s_m)$ be neutral words such that $K(\al)=K(\be)$.
For brevity, set  $\Bx := K(\al) = K(\be)$ and
\[
b_{\al} := \hat{t}_{p_1}(v_{p_1})^*\hat{t}_{q_1}(v_{q_1})\cdots\hat{t}_{p_n}(v_{p_n})^* \hat{t}_{q_n}(v_{q_n})
\qand
b_{\be} := \hat{t}_{r_1}(v_{r_1})^*\hat{t}_{s_1}(v_{s_1})\cdots\hat{t}_{r_m}(v_{r_m})^* \hat{t}_{s_m}(v_{s_m}).
\]
Let $r \in \bo{x}$ and $\xi_r=\mu v_r\in X_{P, r}$ with $\mu \in \bC$.
Then Proposition \ref{P:taut}, along with the equality $\dot{V}_{\al}=\dot{V}_{\be}$, yields
\begin{align*}
t_\ast(b_{\al})t_r(\xi_r)
& =
t_r(\la_\ast(b_{\al})\xi_r) 
=
\mu t_r(W_P^* \dot{V}_{\al}  \de_r)
=
\mu t_r(W_P^* \dot{V}_{\be} \de_r)
=
t_r(\la_\ast(b_{\be})\xi_r) 
=
t_\ast(b_{\be})t_r(\xi_r),
\end{align*}
for the unitary $W_P$ of (\ref{eq:wp un}).
Hence, we have
\[
(t_\ast(b_{\al})-t_\ast(b_{\be}))t_r(\xi_r)=0
\foral
\xi_r \in X_{P,r}, r \in \Bx.
\]
Thus applying condition (ii) for $t$ implies that $t_\ast(b_{\al})-t_\ast(b_{\be})=0$, and therefore
\[
\dot{w}_{\al}=t_\ast(b_{\al})=t_\ast(b_{\be})=\dot{w}_{\be},
\]
 as required.

For condition (T4), let $F$ be a finite set of neutral words, and let $\al$ be a neutral word such that $K(\al)=\bigcup_{\be\in F}K(\be)$.
Let $\F$ be the $\cap$-closure of $\{K(\be) \mid \be\in F\} \cup \{K(\al)\} \cup \{\mt\}$. For each $\mt \not = D \subseteq F$ we write $\be_D$ for the neutral word that arises by concatenating the words $\be \in D$ in some order.
As the induced constructible ideal does not depend on the order of the concatenation we choose (being on neutral words), we have
\[
K(\be_D) = \bigcap_{\be\in D} K(\be),
\]
and
\[
\F=\{K(\be_D):\mt \not = D\subseteq F\}\cup\{K(\al)\}\cup\{\mt\}.
\]
In particular we have
\[
\dot{w}_{\be_D} = \prod_{\be\in D}\dot{w}_{\be}.
\]
For each $\mt \not = D\subseteq F$, let $b_D \in \bo{K}_{K(\be_D),\hat{t}_\ast}$ and $b_{\al} \in \bo{K}_{K(\al), \hat{t}_\ast}$ such that 
\[
t_\ast(b_D)=\dot{w}_{\be_D}
\qand
t_\ast(b_{\al})=\dot{w}_{\al}.
\]
Then, for $r\in K(\al)$ and $\mu v_r \in X_{P,r}$, we have 
\begin{align*}
\left( t_\ast(b_{\al})+\sum_{\substack{\mt \neq D \subseteq F: \\r\in K(\be_D)}} (-1)^{|D|}t_\ast(b_{\be_D}) \right) t_r(\mu v_r)
& =
\left( \dot{w}_{\al}+\sum_{\substack{\mt \neq D \subseteq F: \\r\in K(\be_D)}} (-1)^{|D|}\dot{w}_{\be_D} \right) \mu w_r\\
& =
\left( 1+\sum_{\substack{\mt \neq D \subseteq F: \\r\in \bigcap\limits_{\be\in D} K(\be)}} (-1)^{|D|} \right) \mu w_r\\
& =
\left(1+\sum_{\mt \neq D \subseteq \{\be\in F \mid r\in K(\be) \}} (-1)^{|D|} \right) \mu w_r=0,
\end{align*}
where we used the equalities
\[
\dot{w}_\al w_r = w_r
\qand
\dot{w}_{\be_D} w_r = w_r 
\text{ when }
r \in K(\be_D).
\]
Note also that in the third equality we used that we have
\[
\{\mt \neq D \subseteq F \mid r \in \bigcap\limits_{\be \in D} K(\be)\}
=
\P(\{ \be \in F \mid r \in K(\be) \}) \setminus \{\mt\},
\]
and that these are non-empty sets since $r\in K(\al)=\bigcup_{\be\in F}K(\be)$.
Condition (ii) for $t$ then implies that
\[
t_\ast(b_{\al})+\sum_{\mt \neq D \subseteq F} (-1)^{|D|}t_\ast(b_{\be_D}) = 0,
\]
and therefore we have
\begin{align*}
\prod_{\be\in F}(\dot{w}_{\al}-\dot{w}_{\be})
& =
\dot{w}_{\al}+\sum_{\mt \neq D\subseteq F}(-1)^{|D|} \prod_{\be\in D}\dot{w}_{\be} \\
& =
\dot{w}_{\al}+\sum_{\mt \neq D\subseteq F}(-1)^{|D|} \dot{w}_{\be_D}
=
t_\ast(b_{\al})+\sum_{\mt \neq D \subseteq F} (-1)^{|D|}t_\ast(b_{\be_D}) = 0,
\end{align*}
as required.

For the reverse implication, suppose that $w$ is an equivariant representation of $P$ that satisfies conditions (T1)--(T4) and let $t$ be the unital equivariant representation of $X_P$ associated with $w$.
We will show that $t$ satisfies conditions (i) and (ii) of Theorem \ref{T:Fock cov}.

Let $\bo{x}$ be in $\J$ and pick a neutral word $\al=(p_1,q_1,\dots,p_n,q_n)$ such that $\bo{x} = K(\al)$.
Then condition (T3) implies that
\[
t_\ast(\bo{K}_{\bo{x},\hat{t}_\ast}) = \text{span}\{\dot{w}_{\al}\}.
\]
In particular, combining (T2) with (T3) yields $t_\ast(\bo{K}_{\mt,\hat{t}_\ast})=(0)$, and thus $t$ satisfies condition (i) of Theorem \ref{T:Fock cov}.

For condition (ii), let $\F=\{{\bo{x}_1,\dots,\bo{x}_n}\}$ be a finite $\cap$-closed subset of $\J$ such that $\bigcup_{i=1}^n \bo{x}_i \neq \mt$, and let $b_{\bo{x}_i} \in \bo{K}_{\bo{x}_i,\hat{t}_\ast}$ for $i=1,\dots,n$.
Suppose that 
\[
\sum_{i:r\in\bo{x}_i}t_\ast(b_{\bo{x}_i})t_r(X_{P,r})=0
\foral 
r \in \bigcup_{i=1}^n \bo{x}_i,
\]
and we will show that $\sum_{i=1}^n t_\ast(b_{\bo{x}_i}) = 0$.
For notational convenience, set 
\[
b := \sum_{i=1}^n t_\ast(b_{\bo{x}_i}).
\]
First notice that $t_\ast(\bo{K}_{\mt,\hat{t}_\ast}) = (0)$ as $t$ satisfies condition (i), therefore without loss of generality we may assume that $\Bx_i \neq \mt$ for all $i=1, \dots, n$.
Next let 
\[
F := \{\al_1,\dots,\al_n\}
\]
be a set of neutral words such that $K(\al_i) = \bo{x}_i$ for $i=1, \dots, n$.
Moreover, from the comments above we may pick $\mu_i\in\bC$ such that $t_\ast(b_{\bo{x}_i})=\mu_i\dot{w}_{\al_i}$ for $i=1,\dots,n$. 
As the projections $\dot{w}_{\al_i}$ commute, by the Gel'fand-Naimark Theorem, and considering a unit decomposition (see also \cite[Lemma 3.13]{LS21}), there exists a subset $B$ of $F$ such that
\[
\|b\| = \|Q_B b\|
\qfor
Q_B := \prod_{i:\al_i\in B} \dot{w}_{\al_i} \prod_{j:\al_j\not\in B}(1-\dot{w}_{\al_j}).
\]
Since $b = \sum_{i=1}^n \mu_i \dot{w}_{\al_i}$, by the form of the projection $Q_B$ we have
\[
\|b\| = \|Q_B b\| = \| \sum_{i:\al_i\in B} \mu_i Q_B \| = | \sum_{i:\al_i\in B} \mu_i |.
\]
We consider the following cases.

\smallskip

\noindent
Case 1. Suppose that 
\[
\bigcap_{i:\al_i\in B} K(\al_i) \subseteq \bigcup_{j:\al_j\not\in B}K(\al_j).
\]
Then we have
\[
\bigcap_{i:\al_i\in B}K(\al_i)
=
\bigcup_{j:\al_j\not\in B}\left(K(\al_j)\cap \left(\bigcap_{i:\al_i\in B}K(\al_i)\right)\right).
\]
Applying (T4) for
\[
K(\al)= \bigcap_{i:\al_i\in B}K(\al_i),
\]
where $\al$ is a concatenation of the words $\al_i\in B$, and 
\[
K(\be_j) = K(\al_j)\cap \left(\bigcap_{i:\al_i\in B}K(\al_i)\right)
\text{ such that }
\al_j \notin B,
\] 
where  $\beta_j$ is a concatenation of the words  $\al_i\in B$ and $\al_j$,
gives
\[
\prod_{j : \al_j \notin B} (\dot{w}_\al - \dot{w}_{\be_j}) = 0.
\]
By construction, for any $j$ with $\al_j \notin B$ we have
\[
\dot{w}_{\be_j} \leq \dot{w}_{\al_j}
\qand
\dot{w}_{\be_j} \leq \dot{w}_{\al} = \prod_{i:\al_i\in B}\dot{w}_{\al_i}.
\]
Therefore we have
\[
Q_B = \dot{w}_\al \prod_{j:\al_j\not\in B}(1-\dot{w}_{\al_j}) \leq \dot{w}_\al \prod_{j : \al_j \notin B} (1 - \dot{w}_{\be_j}) = \prod_{j : \al_j \notin B} (\dot{w}_\al - \dot{w}_{\be_j}) = 0,
\]
and thus $b = 0$. 

\smallskip

\noindent
Case 2. Suppose that
\[
\bigcap_{i:\al_i\in B} K(\al_i) \not\subseteq \bigcup_{j:\al_j\not\in B}K(\al_j).
\]
Then there exists a $p\in\bigcap_{i:\al_i\in B} K(\al_i)\subseteq \bigcup_{i=1}^n K(\al_i)$ such that $p\not\in\bigcup_{j:\al_j\not\in B} K(\al_j)$. 
A moment's thought shows that this dichotomy implies that
\[
\{i \mid \al_i \in B\}
=
\{i \mid p \in K(\al_i)\}.
\]
Since $w_p$ is an isometry and $\dot{w}_{\al_i} w_p = w_p$ for all $i$ with $p \in K(\al_i)$, we conclude that
\begin{align*}
\|b\|
& =
|\sum_{i:\al_i\in B} \mu_i|
=
| \sum_{i:p\in K(\al_i)}\mu_i | 
=
\| \sum_{i:p\in K(\al_i)} \mu_i\dot{w}_{\al_i}w_p \|
=
\| \sum_{i:p\in \Bx_i} t_\ast(b_{\bo{x}_i})t_p(v_p) \|
=
0,
\end{align*}
where we applied the assumption for $p \in \bigcup_{i=1}^n \bo{x}_i$ in the last equality.
Thus $b=0$ in this case as well, and the proof is complete.
\end{proof}

\begin{corollary}
Let $P$ be a unital subsemigroup of a discrete group $G$ and let $X_P$ be the associated product system. Then $\T_{\cov}^{\fock}(X_P)$ and $\T_u(P)$ are canonically isomorphic. In particular, a representation of $X_P$ is Fock covariant if and only if the associated semigroup representation of $P$ satisfies conditions (T1)--(T4).
\end{corollary}

\begin{proof}
This is immediate by Proposition \ref{P:fock is u} since the universal Fock covariant representation of $X_P$ is equivariant and the left action on $X_P$ is unital.
\end{proof}

\section{The reduced Hao--Ng isomorphism problem}\label{S:Hao-Ng}

In this section we provide the affirmative solution to the reduced Hao--Ng isomorphism problem for generalised gauge actions by discrete groups.
In order to put the problem into context we require some elements from strong covariant representations and tensor algebras. We will also require some elements from the theory of nonselfadjoint operator algebras.

\subsection{Operator algebras} \label{Ss:opalg}

For the general theory of nonselfadjoint operator algebras and full details, the reader is addressed to \cite{BL04, Pau02}.
By an operator algebra $\fA$ we will mean a norm-closed subalgebra of some $\B(H)$.
Every operator algebra attains a C*-cover, i.e., a completely isometric homomorphism $\io \colon \fA \to \C$ such that $\C = \ca(\io(\fA))$.
The \emph{C*-envelope} $\cenv(\fA)$ is the co-universal C*-cover of $\fA$, i.e., there exists a completely isometric homomorphism $\io \colon \fA \to \cenv(\fA)$ such that for any other C*-cover $j \colon \fA \to \C$ there exists a unique $*$-epimorphism $\Phi \colon \C \to \cenv(\fA)$ such that $\Phi \circ j = \io$.

The existence of the C*-envelope was established by Hamana \cite{Ham79} through the existence of the injective envelope.
An independent proof was established by Dritschel and McCullough \cite{DM05} through the existence of maximal dilations.
Recall that a homomorphism $\phi \colon \fA \to \B(K)$ is called a \emph{dilation} of $\io \colon \fA \to \B(H)$ if $H \subseteq K$ and $\io(\cdot) = P_H \phi(\cdot )|_H$.
A dilation is called \emph{maximal} if it only attains trivial dilations, i.e., dilations by orthogonal summands.
Dritschel and McCullough show in \cite[Theorem 1.2]{DM05} that every (completely isometric) representation of an operator algebra $\fA$ admits a (resp. completely isometric) maximal dilation.
Moreover in \cite[Theorem 4.1]{DM05} they obtain that the C*-algebra of a maximal completely isometric representation satisfies the co-universal property of the C*-envelope.
In a simplified version, Arveson \cite{Arv08} has shown that a representation of $\fA$ is maximal if and only if it has a unique extension to a completely positive map on $\ca(\fA)$ that is a $*$-homomorphism.
The proof in \cite{Arv08} refers to operator systems, but it can be adapted to the operator algebras category; the reader is directed to \cite[Section 2]{DS18} and  \cite[Proposition 4.8]{Sal17} for the full details.

\subsection{Strong covariant representations}

We review the key elements of the strong covariant C*-algebra of a product system established by Sehnem \cite{Seh18, Seh21}, and its Fell bundle considered by Dor-On, the first author, Katsoulis, Laca and Li \cite{DKKLL20}.
The strong covariant algebra of Sehnem \cite{Seh18} as well as the reduced variant considered in \cite{DKKLL20} encodes the previously known constructions for a co-universal or a Cuntz-type C*-algebra, see \cite{DKKLL20, Seh18, Seh21} and the references therein for these connections.

Let $P$ be a unital subsemigroup of a discrete group $G$ and let $X$ be a product system over $P$. Set $A:=X_e$, and for each $p\in P$ let $\vphi_p\colon A\to \L(X_p)$ be the $*$-homomorphism that implements the left action of $A$ on the C*-correspondence $X_p$. For a finite set $F \subseteq G$ let
\[
K_F := \bigcap_{g \in F} gP.
\]
For $r \in P$ and $g \in F$ define the ideal of $A$ given by
\[
I_{r^{-1} K_{\{r,g\}}} :=
\begin{cases}
\bigcap\limits_{s \in K_{\{r,g\}}} \ker \vphi_{r^{-1} s} & \text{if } K_{\{r,g\}} \neq \mt \text{ and } r \notin K_{\{r,g\}},\\
\hspace{.5cm} A & \text{otherwise},
\end{cases}
\]
and set
\[
I_{r^{-1} (r \vee F)} := \bigcap_{g \in F} I_{r^{-1} K_{\{r,g\}}}.
\]
We have $I_{r^{-1}(r \vee F)} = I_{(pr)^{-1}(pr \vee pF)}$ for all $r,p \in P$, and $I_{r^{-1}(r \vee F)} = I_{(s^{-1}r)^{-1}(s^{-1}r \vee s^{-1}F)}$ for all $r \in sP$.
Moreover
\[
I_{r^{-1}(r \vee F_1)} \subseteq I_{r^{-1}(r \vee F_2)} \; \text{ when } \; F_1 \supseteq F_2.
\]
We declare $K_\mt= \mt$ and $I_{r^{-1} (r \vee \mt)} = A$.
For a finite set $F \subseteq G$, let the C*-correspondences
\[
X_F := \sumoplus_{r \in P} X_r I_{r^{-1} (r \vee F)}
\qand
X_F^+ := \sumoplus_{g \in G} X_{gF}.
\]
We declare $X_\mt = X_\mt^+ = \F X$.
Each $X_F^+$ is reducing for the coaction
\[
\de \colon \T_\la(X) \longrightarrow \T_\la(X) \otimes \ca_{\max}(G); \la_p(\xi_p) \mapsto \la_p(\xi_p) \otimes u_p,
\]
giving rise to a $*$-representation
\[
\Phi_F \colon \T_\la(X) \longrightarrow \L( X_F^+ ) ; \la_p(\xi_p) \mapsto (\la_p(\xi_p) \otimes u_p)|_{X_F^+}.
\]
Here we make the identification
\[
X_{gF} \longrightarrow X_{gF} \otimes \de_g ; \xi_r a_r \mapsto \xi_r a_r \otimes \de_g, \text{ for } \xi_r \in X_r, a_r \in I_{r^{-1}(r \vee gF)}.
\]
Moreover $X_F \subseteq X_F^+$ is reducing for $[\T(X)]_e$ and so we obtain the representation
\[
\bigoplus\limits_{\textup{fin } F \subseteq G} \Phi_F(\cdot)|_{X_F} \colon [\T_\la(X)]_e \longrightarrow \prod\limits_{\text{fin } F \subseteq G} \L(X_F).
\]
We fix the ideal $\I_{\scv, e}$ in $[\T(X)]_e$ by using the corona universe, namely
\[
b \in \I_{\scv, e} \qiff \bigoplus\limits_{\textup{fin } F \subseteq G} \Phi_F(\la_\ast(b))|_{X_F} \in c_0 \left( \L(X_F) \mid \textup{fin } F \subseteq G \right).
\]
We write $A \times_X P$ for the equivariant quotient of $\T(X)$ by the induced ideal $\I_{\scv} := \sca{ \I_{\scv, e} }$.
It is shown in \cite{Seh18} that this construction does not depend on $G$.
Moreover, we define the \emph{strong covariant} bundle of $X$ to be the Fell bundle
\[
\S\C X := \Big\{ [A \times_X P]_g \Big\}_{g \in G},
\]
given by the coaction by $G$ on $A \times_X P$.
A representation of $X$ that promotes to a representation of $\S\C X$ will be called a \emph{strong covariant representation of $X$}.
By combining Proposition \ref{P:qnt bnd} with Proposition \ref{P:max psx} we have that $A \times_X P$ is the universal C*-algebra of $\S\C X$.
We further write $A \times_{X, \la} P$ for the reduced C*-algebra of $\S\C X$.
As the strong covariant relations are graded and induced by representations on $\T_\la(X)$ it follows that a strong covariant representation is Fock covariant.

For a product system $X$ we will write $\T_\la(X)^+$ for the \emph{tensor algebra of $X$}, i.e., for the norm-closed subalgebra generated by $\{\la_p(X_p)\}_{p \in P}$ inside $\T_\la(X)$.
We will require the following consequence of \cite[Corollary 3.5]{Seh21}.
We note that although only product systems in the sense of Fowler are considered in \cite{Seh21}, the arguments apply in our setting as well.

\begin{corollary}\label{C:cis}
Let $P$ be a unital subsemigroup of a discrete group $G$ and let $X$ be a product system over $P$.
If $t$ is a Fock covariant injective representation of $X$ and $t$ admits a coaction by $G$ which is normal, then the map
\[
\T_{\la}(X)^+ \to \ol{\alg}\{t_p(X_p) \mid p \in P\}; \la_p(\xi_p) \mapsto t_p(\xi_p),
\] 
is a (well-defined) completely isometric isomorphism.
\end{corollary}

\begin{proof}
Let the canonical $*$-epimorphism $\Phi \colon \T_{\cov}^{\fock}(X) \to \ca(t)$.
Since $t$ is equivariant, then $\Phi$ induces a $*$-epimorphism between the reduced C*-algebras of the induced Fell bundles.
However, since the coaction on $\ca(t)$ is normal, it follows that $\ca(t)$ coincides with the reduced C*-algebra of its induced Fell bundle.
Hence there exists a canonical $*$-epimorphism $\Phi_\la \colon \T_\la(X) \to \ca(t)$, which is injective on $A$ since $t$ is injective.
By \cite[Corollary 3.5]{Seh21} it follows that the restriction of $\Phi_\la$ on $\T_\la(X)^+$ is completely isometric.
\end{proof}

Since $A \times_{X,\la} P$ satisfies the conditions of Corollary \ref{C:cis}, it is a C*-cover of $\T_\la(X)^+$.
In \cite[Theorem 5.1]{Seh21}, Sehnem establishes that $A \times_{X, \la} P$ is the C*-envelope of $\T_\la(X)^+$.

\subsection{Crossed products}

The theory of crossed products of C*-algebras is well-known, see for example \cite{Wil07}.
Katsoulis and Ramsey \cite{KR16} have extended this to group actions over possibly nonselfadjoint operator algebras.
Here we will comment only on the parts that are relevant to the reduced Hao--Ng isomorphism problem.

Let $\fH$ be a locally compact group that acts on an operator algebra $\fA$ by completely isometric automorphisms.
The action of $\fH$ then extends also to a group action $\dot{\al}$ of the C*-envelope of $\fA$, and we can form the reduced crossed product $\cenv(\fA) \rtimes_{\dot\al, \la} \fH$.
By considering the copy of $\fA$ inside $\cenv(\fA)$, the reduced crossed product $\fA \rtimes_{\al, \la} \fH$ is defined as the norm-closed subalgebra of $\cenv(\fA) \rtimes_{\dot{\al}, \la} \fH$ generated by the $\fA$-valued functions \cite[Definition 3.17]{KR16}.
One of the main questions in the theory is whether this inclusion induces a canonical $*$-isomorphism with the C*-envelope of $\fA \rtimes_{\al, \la} \fH$, i.e., whether we have
\begin{equation} \label{eq:cenv}
\cenv(\fA \rtimes_{\al, \la} \fH) \stackrel{?}{\simeq} \cenv(\fA) \rtimes_{\dot\al, \la} \fH.
\end{equation}
When $\fA$ admits a contractive approximate identity, this has been answered to the affirmative when $\fH$ is discrete \cite{Kat17}, and when $\fH$ is abelian \cite{KR16}.
By using the maximal representations of \cite{DM05} we can remove the contractive approximate identity hypothesis when $\fH$ is discrete.

\begin{proposition} \label{P:dis max}
Let $\fH$ be a discrete group acting by $\al$ on an operator algebra $\fA$.
Then
\[
\cenv(\fA \rtimes_{\al, \la} \fH) \simeq \cenv(\fA) \rtimes_{\dot\al, \la} \fH.
\]
\end{proposition}

\begin{proof}
We identify $\fA$ with its copy in $\cenv(\fA)$. Let $\pi \colon \fA \to \B(H)$ be a maximal completely isometric representation of $\fA$.
By \cite{DM05} (see also the relevant comments in Subsection \ref{Ss:opalg}) the map $\pi$ has a unique extension to a faithful $*$-representation on $\cenv(\fA)$, denoted by the same symbol.
Therefore $\ol{\pi} \rtimes U$ is a faithful $*$-representation of $\cenv(\fA) \rtimes_{\dot\al, \la} \fH$, and thus its restriction on $\fA \rtimes_{\al, \la} \fH$ is completely isometric.
By \cite{DM05} (see also the relevant comments in Subsection \ref{Ss:opalg}), it suffices to show that $\ol{\pi} \rtimes U$ is maximal on $\fA \rtimes_{\al, \la} \fH$, as then it will follow that
\[
\cenv(\fA \rtimes_{\al, \la} \fH) \simeq \ca(\ol{\pi} \rtimes U) \simeq \cenv(\fA) \rtimes_{\dot\al, \la} \fH
\]
by $*$-isomorphisms fixing $\fA \rtimes_{\al, \la} \fH$.

Towards this end, let $\rho \colon \fA \rtimes_{\al, \la} \fH \to \B(K)$ be a maximal dilation of $\ol{\pi} \rtimes U|_{\fA \rtimes_{\al, \la} \fH}$, 
and let us denote by the same symbol the unique extension of $\rho$ to a $*$-representation on 
\[
\ca(\fA \rtimes_{\al, \la} \fH) 
= 
\ol{\spn}\{\ol{\pi}(c) U_{\fh} \mid c \in \cenv(\fA), \fh \in \fH\}.
\]
By Arveson's Extension Principle, each $\pi \circ \al_\fh$ is a maximal representation of $\fA$, and it is a standard argument that $\ol{\pi}$ is maximal as the discrete direct sum of maximal representations, e.g., see \cite[Proposition 4.4]{Arv11}.
Therefore $\rho$ on $\ol{\pi}(\fA)$ takes up the form
\[
\rho(\ol{\pi}(a))
=
\begin{bmatrix} \ol{\pi}(a) & 0 \\ 0 & \si(a) \end{bmatrix} \foral a \in \fA,
\]
for a representation $\si$ of $\fA$.
Next consider $\ol{\pi}(a) U_{\fh}$ for $a \in \fA$ and $\fh \in \fH$, and write
\[
\rho(\ol{\pi}(a) U_\fh)
=
\begin{bmatrix} \ol{\pi}(a) U_{\fh} & x \\ y & z \end{bmatrix}.
\]
By using the unique extension property of $\rho$ we have
\begin{align*}
\begin{bmatrix} \ol{\pi}(a) \ol{\pi}(a)^* + xx^* & \ast \\ \ast & \ast\end{bmatrix}
& =
\rho(\ol{\pi}(a) U_\fh) \rho(\ol{\pi}(a) U_\fh)^*
=
\rho(\ol{\pi}(a) U_\fh U_{\fh}^* \ol{\pi}(a)^*) \\
& =
\rho(\ol{\pi}(a) \ol{\pi}(a)^*)
=
\rho(\ol{\pi}(a)) \rho(\ol{\pi}(a))^*
=
\begin{bmatrix} \ol{\pi}(a) \ol{\pi}(a)^* & 0 \\ 0 & \si(a) \si(a)^* \end{bmatrix}.
\end{align*}
By equating the $(1,1)$-entries we get $x = 0$.
On the other hand, by using the covariance in the C*-crossed product we can write
\[
\ol{\pi}(\al_\fh^{-1}(a))^* \ol{\pi}(\al_{\fh}^{-1}(a))
=
\ol{\pi}(\al_\fh^{-1}(a))^* U_{\fh}^* U_{\fh} \ol{\pi}(\al_{\fh}^{-1}(a))
=
U_\fh^* \ol{\pi}(a)^*  \ol{\pi}(a) U_{\fh}.
\]
Then a similar computation gives
\begin{align*}
\begin{bmatrix} \ol{\pi}(\al_\fh^{-1}(a))^* \ol{\pi}(\al_{\fh}^{-1}(a)) + y^*y & \ast \\ \ast & \ast\end{bmatrix}
& =
\rho(\ol{\pi}(a) U_\fh)^*  \rho(\ol{\pi}(a) U_\fh) 
 =
\rho \left( U_\fh^* \ol{\pi}(a)^*  \ol{\pi}(a) U_{\fh} \right) \\
& =
\rho \left( \ol{\pi}(\al_\fh^{-1}(a))^* \ol{\pi}(\al_{\fh}^{-1}(a)) \right) 
 =
\rho\left( \ol{\pi}(\al_\fh^{-1}(a)) \right)^* \rho \left( \ol{\pi}(\al_{\fh}^{-1}(a)) \right) \\
& =
\begin{bmatrix} \ol{\pi}(\al_\fh^{-1}(a))^* \ol{\pi}(\al_\fh^{-1}(a)) & 0 \\ 0 & \si(\al_\fh^{-1}(a))^* \si(\al_\fh^{-1}(a)) \end{bmatrix}.
\end{align*}
By equating the $(1,1)$-entries we get $y = 0$.
Therefore $\rho$ is a trivial dilation, as required.
\end{proof}

\subsection{Hao--Ng}

Let us return to the product system discussion.
Let $P$ be a unital subsemigroup of a discrete group $G$ and let $X$ be a product system over $P$.
Let $\fH$ be a discrete group, and let $\al$ be a \emph{generalised gauge action} of $\fH$ on the Fock C*-algebra $\T_\la(X)$, i.e., every $\la_p(X_p)$ is $\al$-invariant in the sense that
\[
\al_{\fh}( \la_p(X_p) ) = \la_p(X_p) \foral p \in P\text{ and } \fh\in \fH.
\]
Let $\pi$  be a faithful $*$-representation of $\T_\la(X)$ in some $\B(H)$ and consider the faithful $*$-representa\-tion $\ol{\pi} \rtimes U$ on $\ell^2(\fH, H)$ that gives rise to $\T_\la(X) \rtimes_{\al, \la} \fH$.
Since $\al$ restricts to an action on $\T_\la(X)^+$, from \cite[Corollary 3.16]{KR16} we have a canonical completely isometric copy of the crossed product $\T_\la(X)^+ \rtimes_{\al, \la} \fH$ in $\T_\la(X) \rtimes_{\al, \la} \fH$.
We can now consider the family $X \rtimes_{\al, \la} \fH$ of the subspaces
\[
(X \rtimes_{\al, \la} \fH)_p 
:= 
\ol{\spn} \{ \ol{\pi}(\la_p(\xi_p)) U_{\fh} \mid \xi_p \in X_p, \fh \in \fH\}
\text{ for }
p \in P.
\]
It is accustomed to use the short form $X_p \rtimes_{\al, \la} \fH$ instead of $(X \rtimes_{\al, \la} \fH)_p$.
We will show that this construction gives rise to a product system.

\begin{proposition} \label{P:cpps}
Let $P$ be a unital subsemigroup of a discrete group $G$ and let $X$ be a product system over $P$.
Let $\al$ be a generalised gauge action of a discrete group $\fH$ on $\T_\la(X)$.
Then $X \rtimes_{\al, \la} \fH$ is a product system over $P$.
\end{proposition}

\begin{proof}
For condition (i) by definition we have
\[
X_e \rtimes_{\al, \la} \fH 
:= 
\ol{\spn} \{ \ol{\pi}(\la_e(\xi_e)) U_{\fh} \mid \xi_e \in X_e, \fh \in \fH\} 
\simeq
A \rtimes_{\al, \la} \fH,
\]
which is a C*-algebra.
The other two conditions follow by the fact that we have the equality $\al_{\fh}(\la_p(X_p)) = \la_p(X_p)$ for all $p \in P$ and $\fh\in \fH$, the covariant relations of the crossed product and that $\la$ is a product system representation.
Indeed, for condition (ii) we have
\begin{align*}
\ol{\pi}(\la_p(X_p)) U_{\fh_1} \ol{\pi}(\la_q(X_q)) U_{\fh_2}
& =
\ol{\pi}\left( \la_p(X_p) \al_{\fh_1}(\la_q(X_q)) \right) U_{\fh_1 \fh_2} \\
& =
\ol{\pi} \left( \la_p(X_p) \la_q(X_q) \right) U_{\fh_1 \fh_2}
\subseteq
\ol{\pi} \left( \la_{pq}(X_p X_q) \right) U_{\fh_1 \fh_2}
\subseteq
X_{pq} \rtimes_{\al, \la} \fH.
\end{align*}
By considering finite linear combinations and their norm-limits we obtain
\[
(X_p \rtimes_{\al, \la} \fH) \cdot (X_q \rtimes_{\al, \la} \fH) \subseteq X_{pq} \rtimes_{\al, \la} \fH.
\]
For condition (iii) we have
\begin{align*}
\left(\ol{\pi}(\la_p(X_p)) U_{\fh_1} \right)^* \ol{\pi}(\la_{pq}(X_{pq})) U_{\fh_2}
& =
\ol{\pi} \circ \al_{\fh_1}^{-1} \left( \la_p(X_p)^* \la_{pq}(X_{pq}) \right) U_{\fh_1^{-1} \fh_2} \\
& \subseteq
\ol{\pi} \left( \la_q(X_q) \right) U_{\fh_1^{-1} \fh_2}
\subseteq
X_{q} \rtimes_{\al, \la} \fH.
\end{align*}
By considering finite linear combinations and their norm-limits we obtain
\[
(X_p \rtimes_{\al, \la} \fH)^* \cdot (X_{pq} \rtimes_{\al, \la} \fH) \subseteq X_{q} \rtimes_{\al, \la} \fH,
\]
and the proof is complete.
\end{proof}

In order to facilitate comparisons we will use the superscript $\rtimes$ for the representations of $X\rtimes_{\al,\la}\fH$.
That is, we write $\la^\rtimes$ for the Fock representation of $X \rtimes_{\al,\la} \fH$, and we fix a representation $\hat{t}^\rtimes$ of $X\rtimes_{\al,\la}\fH$ such that
\[
\T(X\rtimes_{\al,\la}\fH) = \ca(\hat{t}^\rtimes).
\]
We will also write $\io^\rtimes$ for the identity representation $X \rtimes_{\al, \la} \fH \hookrightarrow \T_\la(X) \rtimes_{\al, \la} \fH$.

Due to the properties of the generalised gauge action we have a canonical identification of the $\bo{K}_{\bullet}$-cores.
Note that, since every $\la_p(X_p)$ is $\al$-invariant, then so is every C*-subalgebra $\bo{K}_{\Bx, \la_\ast}$ for $\Bx \in \J$.
By construction we have
\[
\bo{K}_{\Bx, \la_\ast} \rtimes_{\al, \la} \fH \subseteq \T_\la(X) \rtimes_{\al, \la} \fH
\foral
\Bx \in \J.
\]

\begin{proposition} \label{P:alg}
Let $P$ be a unital subsemigroup of a discrete group $G$ and let $X$ be a product system over $P$.
Let $\al$ be a generalised gauge action of a discrete group $\fH$ on $\T_\la(X)$ and let $\io^\rtimes \colon X\rtimes_{\al,\la}\fH \to \T_{\la}(X) \rtimes_{\al,\la} \fH$ be the identity representation.
If $\Bx \in \J$, then
\[
\io^\rtimes_\ast(\bo{K}_{\Bx, \hat{t}^\rtimes_\ast})
=
\bo{K}_{\Bx, \io^\rtimes_\ast}
=
\bo{K}_{\Bx, \la_\ast} \rtimes_{\al, \la} \fH.
\]
Consequently, if $E_\fH \colon \T_\la(X) \rtimes_{\al,\la} \fH \to \T_\la(X)$ is the faithful conditional expectation of the reduced crossed product, then
\[
E_{\fH}( \io^\rtimes_\ast(b_{\Bx}) ) \in \bo{K}_{\Bx, \la_\ast}
\foral
b_{\Bx} \in \bo{K}_{\Bx, \hat{t}^\rtimes_\ast}.
\]
\end{proposition}

\begin{proof}
We will make the standard identification of $\T_\la(X)$ with $\ol{\pi}(\T_\la(X))$, and of $\T_\la(X) \rtimes_{\al,\la} \fH$ with $\ca(\ol{\pi}, U)$.
In order to make a distinction we write $e_G$ for the unit of $G$ and $e_\fH$ for the unit of $\fH$.
For the first part note that we have $\io^\rtimes_\ast(\bo{K}_{\Bx, \hat{t}^\rtimes_\ast}) = \bo{K}_{\Bx, \io^\rtimes_\ast}$ since $\io^\rtimes$ is a representation of $X \rtimes_{\al, \la} \fH$.
Moreover by checking the generators we obtain
\[
\bo{K}_{\Bx, \la_\ast} \rtimes_{\al, \la} \fH 
\subseteq 
\bo{K}_{\Bx, \io^\rtimes_\ast}.
\]
For the reverse inclusion, let $p_1, q_1, \dots, p_n, q_n\in P$ and $\eps, \eps' \in\{0,1\}$ such that $p_1^{-\eps} q_1 \cdots p_n^{-1} q_n^{\eps'} =e_G$ and $q_n^{-\eps'} p_n \dots q_1^{-1} p_1^{\eps} P = \Bx$, and consider an element $b_\Bx$ in $\bo{K}_{\Bx,\hat{t}^\rtimes_\ast}$ of the form:
\[
(\hat{t}^\rtimes \left(\ol{\pi}\left(\la_{p_1}(\xi_{p_1})\right)U_{\fh_1}\right)^*)^{\eps}
\hat{t}^\rtimes(\ol{\pi}\left(\la_{q_1}(\xi_{q_1})\right)U_{\fh_2}) \cdots 
\hat{t}^\rtimes(\ol{\pi}\left(\la_{p_n}(\xi_{p_n})\right)U_{\fh_{2n-1}})^*
 \hat{t}^\rtimes(\ol{\pi}\left(\la_{q_n}(\xi_{q_n})\right)U_{\fh_{2n}}) ^{\eps'}.
\]
By the covariance relation of $(\ol{\pi}, U)$, and the fact that we have the equality $\al_\fh(\la_p(X_p)) = \la_p(X_p)$ for all $p \in P$ and $\fh \in \fH$, we can write
\begin{align*}
\io^\rtimes_\ast(b_{\Bx})
& =
 (\left(\ol{\pi}\left(\la_{p_1}(\xi_{p_1})\right)U_{\fh_1}\right)^*)^{\eps}
\ol{\pi}\left(\la_{q_1}(\xi_{q_1})\right)U_{\fh_2} \cdots 
(\ol{\pi}\left(\la_{p_n}(\xi_{p_n})\right)U_{\fh_{2n-1}})^*
(\ol{\pi}\left(\la_{q_n}(\xi_{q_n})\right)U_{\fh_{2n}})^{\eps'}\\
&=
 \ol{\pi}\left(\la_{p_1}(\eta_{p_1})^*\right)^{\eps}
\ol{\pi}\left(\la_{q_1}(\eta_{q_1})\right) \cdots
\ol{\pi}\left(\la_{p_n}(\eta_{p_n})^*\right)
\ol{\pi}\left(\la_{q_n}(\eta_{q_n})\right)^{\eps'} U_{\fh} \\
& =
\ol{\pi}\left( \left(\la_{p_1}(\eta_{p_1})^* \right)^{\eps} \la_{q_1}(\eta_{q_1}) \cdots \la_{p_n}(\eta_{p_n})^* \la_{q_n}(\eta_{q_n})^{\eps'} \right) U_{\fh}
\end{align*}
for $\fh := \fh_1^{-\eps} \fh_2 \cdots \fh_{2n-1}^{-1} \fh_{2n}^{\eps'}$, and some appropriate choice of $\eta_{p_i} \in X_{p_i}$ and $\eta_{q_i}\in X_{q_i}$ for $i=1,\dots,n$.
Therefore we have
\[
\io^\rtimes_\ast(b_{\Bx}) \in \ol{\pi}(\bo{K}_{\Bx, \la_\ast}) \cdot U_{\fh} \subseteq \bo{K}_{\Bx, \la_\ast} \rtimes_{\al, \la} \fH.
\]
By considering linear combinations and their norm-limits we have
\[
\bo{K}_{\Bx, \io^\rtimes_\ast} \subseteq \bo{K}_{\bo{x}, \la_\ast} \rtimes_{\al, \la} \fH,
\]
as required.

For the second part, let $p_1, q_1, \dots, p_n, q_n\in P$ and $\eps, \eps' \in \{0,1\}$ such that $p_1^{-\eps} q_1 \cdots p_n^{-1} q_n^{\eps'} = e_G$ and $q_n^{-\eps'} p_n \dots q_1^{-1} p_1^{\eps} P = \Bx$, and consider an element $b_\Bx$ in $\bo{K}_{\Bx,\hat{t}^\rtimes_\ast}$ of the form:
\[
(\hat{t}^\rtimes \left(\ol{\pi}\left(\la_{p_1}(\xi_{p_1})\right)U_{\fh_1}\right)^*)^{\eps}
\hat{t}^\rtimes(\ol{\pi}\left(\la_{q_1}(\xi_{q_1})\right)U_{\fh_2}) \cdots 
\hat{t}^\rtimes(\ol{\pi}\left(\la_{p_n}(\xi_{p_n})\right)U_{\fh_{2n-1}})^*
\hat{t}^\rtimes(\ol{\pi}\left(\la_{q_n}(\xi_{q_n})\right)U_{\fh_{2n}})^{\eps'}.
\]
Due to the crossed product covariance and the definition of the action we can write
\[
\io^\rtimes_\ast(b_{\Bx})
= 
\left( \ol{\pi}\left(\la_{p_1}(\eta_{p_1}) \right)^* \right)^{\eps}
\ol{\pi}\left(\la_{q_1}(\eta_{q_1})\right) 
\cdots 
\ol{\pi}\left(\la_{p_n}(\eta_{p_n})\right)^* 
\ol{\pi}\left(\la_{q_n}(\eta_{q_n})\right)^{\eps'} U_{\fh}
\]
for $\fh := \fh_1^{-\eps} \fh_2 \cdots \fh_{2n-1}^{-1} \fh_{2n}^{\eps'}$, and some $\eta_{p_i} \in X_{p_i}$ and $\eta_{q_i}\in X_{q_i}$ for $i=1,\dots, n$.
By definition we have
\[
E_{\fH}(\io^\rtimes_\ast(b_{\Bx}))
=
\begin{cases}
 \ol{\pi}\left(\left(\la_{p_1}(\eta_{p_1}) ^* \right)^{\eps}
\la_{q_1}(\eta_{q_1})
\cdots 
\la_{p_n}(\eta_{p_n})^* 
\la_{q_n}(\eta_{q_n})^{\eps'}\right)
& \text{if } \fh = e_{\fH}, \\
\hspace{.1cm} 0 & \text{otherwise},
\end{cases}
\]
and thus $E_{\fH}(\io^\rtimes_\ast(b_{\Bx})) \in \ol{\pi}\left(\bo{K}_{\Bx, \la_\ast}\right)$ as required.
The proof is completed by considering finite linear combinations of elements of this form in $\bo{K}_{\Bx, \hat{t}^\rtimes_\ast}$, and their norm-limits.
\end{proof}

The following proposition is the key result of this section.

\begin{proposition} \label{P:id cp}
Let $P$ be a unital subsemigroup of a discrete group $G$ and let $X$ be a product system over $P$.
Let $\al$ be a generalised gauge action of a discrete group $\fH$ on $\T_\la(X)$.
Then the identity representation $\io^\rtimes \colon X\rtimes_{\al,\la}\fH \to \T_{\la}(X) \rtimes_{\al,\la} \fH$ is equivariant, Fock covariant, and injective.
\end{proposition}

\begin{proof}
We will show that the identity representation satisfies the conditions of Theorem \ref{T:Fock cov}.
By definition it is injective.
Moreover, by \cite[Lemma 7.16]{Wil07} we have
\[
 (\T_\la(X) \otimes \ca_{\max}(G)) \rtimes_{\al\otimes \id, \la} \fH
 \simeq 
 (\T_\la(X) \rtimes_{\al, \la} \fH) \otimes \ca_{\max}(G),
\]
by a canonical $*$-isomorphism, and hence we get that $\io^\rtimes$ is equivariant.
It remains to show that $\io^\rtimes$ satisfies conditions (i) and (ii) of Theorem \ref{T:Fock cov}.
We will make the standard identification of $\T_\la(X)$ with $\ol{\pi}(\T_\la(X))$, and of $\T_\la(X) \rtimes_{\al,\la} \fH$ with $\ca(\ol{\pi}, U)$.

For condition (i), an application of Proposition \ref{P:alg} and using that $\bo{K}_{\mt, \la_\ast} = (0)$ yield
\[
\bo{K}_{\mt, \io^\rtimes_\ast} = \bo{K}_{\mt, \la_\ast} \rtimes_{\al, \la} \fH = (0).
\]

For condition (ii), let $\F = \{\Bx_1, \dots, \Bx_n\}$ be a finite $\cap$-closed subset of $\J$  such that $\bigcup_{i=1}^n \bo{x}_i \neq \mt$ and let $b_{\Bx_i} \in \bo{K}_{\Bx_i, \hat{t}^{\rtimes}_\ast}$ for $i=1, \dots, n$ that satisfy
\[
\sum_{i : r \in \Bx_i} \io^\rtimes_\ast(b_{\Bx_i}) \io^\rtimes_r\left(\ol{\pi}\left(\la_r(X_r\right)U_{\fh}\right) = (0)
\foral r \in \bigcup_{i=1}^n \Bx_i\text{ and }\fh \in \fH.
\]
We will show that $\sum_{i=1}^n \io^\rtimes_\ast(b_{\Bx_i}) = 0$.
For notational convenience, set
\[
c := \sum_{i=1}^n \io^\rtimes_\ast(b_{\Bx_i})
\qand
c_r:=\sum_{i:r\in \bo{x}_i} \io^\rtimes_\ast (b_{\bo{x}_i})
\text{ for each }
r \in \bigcup_{i=1}^n \Bx_i.
\]
Since $b_{\bo{x}_i}^* b_{\bo{x}_j} \in \bo{K}_{\Bx_i \cap \Bx_j, \hat{t}^\rtimes_\ast}$, by Proposition \ref{P:alg} we have
\[
E_{\fH} \left( \io^\rtimes_\ast(b_{\bo{x}_i}^* b_{\bo{x}_j}) \right) 
\in
\ol{\pi}(\bo{K}_{\Bx_i \cap \Bx_j, \la_\ast}),
\]
and in particular we can write 
\[
E_{\fH} \left( \io^\rtimes_\ast(b_{\bo{x}_i}^*b_{\bo{x}_j}) \right) 
= 
\ol{\pi}\left(\la_\ast (d_{\bo{x}_i \cap\bo{x}_j})\right),
\text{ for some }
d_{\bo{x}_i\cap\bo{x}_j}\in \bo{K}_{\bo{x}_i \cap\bo{x}_j, \hat{t}_\ast}.
\] 
Therefore we have
\[
E_{\fH}(c^*c)
=
\sum_{i,j=1}^n E_{\fH}(\io^\rtimes_\ast(b_{\bo{x}_i}^* b_{\bo{x}_j})) 
=
\ol{\pi}\left(\sum_{i,j=1}^n \la_\ast (d_{\bo{x}_i \cap\bo{x}_j})\right)
\in 
\ol{\pi}\left(\bo{B}_{\F,  \la_\ast}\right),
\]
where we used that $\{\bo{x}_i\cap\bo{x}_j : i,j=1,\dots,n\}=\F$ for the inclusion, since $\F$ is $\cap$-closed.
Now let
\[
r\in \bigcup_{i,j=1}^n (\bo{x}_i\cap\bo{x}_j)
=
\bigcup_{i=1}^n \bo{x}_i.
\]
By using that $\F$ is $\cap$-closed we obtain
\[
\{(i,j)\in [n]^2 \mid r\in\bo{x}_i\cap\bo{x}_j\}
=
\{i\in [n] \mid r\in \bo{x}_i\} \times \{j \in [n] \mid r\in \bo{x}_j\},
\]
and hence
\begin{align*}
E_{\fH}(c_r^* c_r)
& =
\sum_{i: r \in \bo{x}_i} \sum_{j: r \in \bo{x}_j} E_{\fH} \left( \io^\rtimes_\ast(b_{\bo{x}_i}^* b_{\bo{x}_j}) \right) 
=
\sum_{i,j:r\in \bo{x}_i\cap\bo{x}_j} E_{\fH} \left( \io^\rtimes_\ast(b_{\bo{x}_i}^* b_{\bo{x}_j}) \right) 
=
\sum_{i,j:r\in \bo{x}_i\cap\bo{x}_j} \ol{\pi}\left( \la_\ast(d_{\bo{x}_i\cap\bo{x}_j})\right).
\end{align*}
Since $r \in \bigcup_{i=1}^n \bo{x}_i$, by using the assumption we have $c_r\ol{\pi}(\la_r(\xi_r))=0$ for all $\xi_r \in X_r$.
Thus we get
\begin{align*}
\ol{\pi}(\la_r(\xi_r))^*E_{\fH}(c_r^*c_r) \ol{\pi}(\la_r(\xi_r))
=
E_{\fH}\left( \ol{\pi}(\la_r(\xi_r))^*c_r^*c_r\ol{\pi}(\la_r(\xi_r)) \right)
=
0,
\end{align*}
where we used that $\T_\la(X)$ is in the multiplicative domain of $E_\fH$.
Consequently, we obtain
\[
E_{\fH}(c_r^*c_r)^{1/2} \ol{\pi}(\la_r(\xi_r))=0
\]
and in particular
\[
\ol{\pi}\left(\sum_{i,j: r\in\bo{x}_i\cap\bo{x}_j} \la_\ast(d_{\bo{x}_i\cap\bo{x}_j})\la_r(\xi_r)\right)
=
\sum_{i,j: r\in\bo{x}_i\cap\bo{x}_j}\ol{\pi}\left( \la_\ast(d_{\bo{x}_i\cap\bo{x}_j})\right)\ol{\pi}(\la_r(\xi_r))
=
E_{\fH}(c_r^*c_r) \ol{\pi}(\la_r(\xi_r))=0.
\]
Since $\ol{\pi}$ is faithful we deduce that
\[
\sum_{i,j: r\in\bo{x}_i\cap\bo{x}_j} \la_\ast(d_{\bo{x}_i\cap\bo{x}_j}) \la_r(X_r)
=
(0)
\foral
r\in \bigcup_{i,j=1}^n (\bo{x}_i\cap\bo{x}_j).
\]
Since $\la$ satisfies condition (ii) of Theorem \ref{T:Fock cov} we get $\sum_{i,j=1}^n \la_\ast(d_{\bo{x}_i\cap\bo{x}_j}) = 0$, and therefore
\[
E_{\fH}(c^*c)
=
\ol{\pi} \left( \sum_{i,j=1}^n \la_\ast(d_{\bo{x}_i\cap\bo{x}_j}) \right) 
=
0.
\]
Faithfulness of $E_{\fH}$ then implies that $c=0$, as required.
\end{proof}

Since $A \times_{X, \la} P$ is the C*-envelope of $\T_\la(X)^+$ it inherits a group action $\dot{\al}$ from $\fH$.
The reduced Hao--Ng isomorphism problem asks if there exists a canonical $*$-isomorphism such that
\[
(A \rtimes_{\al, \la} \fH) \times_{X \rtimes_{\al, \la} \fH,\la} P
\stackrel{?}{\simeq}
(A \times_{X, \la} P) \rtimes_{\dot \al, \la} \fH.
\]
By \cite[Theorem 5.1]{Seh21} this is equivalent to asking for a canonical $*$-isomorphism
\begin{equation} \label{eq:cenv}
\cenv(\T_\la(X \rtimes_{\al, \la} \fH)^+) \stackrel{?}{\simeq} \cenv(\T_\la(X)^+) \rtimes_{\dot\al, \la} \fH,
\end{equation}
which by Proposition \ref{P:dis max}  is equivalent to the existence of a canonical completely isometric isomorphism
\begin{equation} \label{eq:tensor}
\T_\la(X \rtimes_{\al, \la} \fH)^+ \simeq \T_\la(X)^+ \rtimes_{\al, \la} \fH.
\end{equation}
A careful investigation of the arguments in \cite{DKKLL20, DK20, Kat17, Kat20, KR21} suggests that it is enough to show that the embedding $X \rtimes_{\al, \la} \fH \hookrightarrow \T_\la(X) \rtimes_{\al, \la} \fH$ is Fock covariant, without passing through a $*$-isomorphism of their ambient C*-algebras.

\begin{theorem} \label{T:hao-ng}
Let $P$ be a unital subsemigroup of a discrete group $G$ and let $X$ be a product system over $P$.
Let $\al$ be a generalised gauge action of a discrete group $\fH$ on $\T_\la(X)$.
Then the identity representation $\io^\rtimes \colon X\rtimes_{\al,\la}\fH \to \T_{\la}(X) \rtimes_{\al,\la} \fH$ lifts to a completely isometric isomorphism
\[
\T_\la(X \rtimes_{\al, \la} \fH)^+ \simeq \T_\la(X)^+ \rtimes_{\al, \la} \fH.
\]
Consequently, the reduced Hao--Ng isomorphism problem has an affirmative answer, i.e.,
\[
(A \rtimes_{\al, \la} \fH) \times_{X \rtimes_{\al, \la} \fH,\la} P
\simeq
(A \times_{X, \la} P) \rtimes_{\dot \al, \la} \fH,
\]
by a canonical $*$-isomorphism, where $\dot{\al}$ is the induced action of $\fH$ on $A \times_{X,\la} P$.
\end{theorem}

\begin{proof}
By Proposition \ref{P:id cp} the identity representation is an equivariant Fock covariant injective representation of $X\rtimes_{\al,\la}\fH$.
Moreover, by \cite[Lemma 7.16]{Wil07} we have
\[
 (\T_\la(X) \otimes \ca_{\la}(G)) \rtimes_{\al\otimes \id, \la} \fH
 \simeq 
 (\T_\la(X) \rtimes_{\al, \la} \fH) \otimes \ca_{\la}(G),
\]
by a canonical $*$-isomorphism, and hence we get that the coaction that $\io^\rtimes$ admits is normal.
Since
\[
\T_\la(X)^+ \rtimes_{\al, \la} \fH
=
\ol{\alg}\{\io^\rtimes_p(X_p \rtimes_{\al, \la} \fH) \mid p \in P\},
\]
the required completely isometric isomorphism is induced by Corollary \ref{C:cis}.
By Proposition \ref{P:dis max} we then have
\[
\cenv(\T_\la(X \rtimes_{\al, \la} \fH)^+) \simeq \cenv(\T_\la(X)^+) \rtimes_{\dot{\al}, \la} \fH,
\]
and \cite[Theorem 5.1]{Seh21} finishes the proof.
\end{proof}


\end{document}